\documentclass[12pt]{article}

\usepackage{amsmath,amssymb,amsthm,lscape,young,hyperref,cancel,makeidx,enumitem}
 \usepackage{verbatim}
\usepackage{color}

 \usepackage[normalem]{ulem}
 
\usepackage{cancel}

\usepackage{geometry}
\usepackage{graphicx}
\usepackage{epstopdf}
\usepackage{mathabx}
\usepackage[all,arc]{xy}

\newtheorem{theorem}{Theorem}[section]
\newtheorem{proposition}[theorem]{Proposition}
\newtheorem{lemma}[theorem]{Lemma}

\theoremstyle{definition}
\newtheorem{definition}[theorem]{Definition}

\theoremstyle{remark}
\newtheorem{remark}[theorem]{Remark}

\numberwithin{theorem}{section}

\newcommand{\ii}{\mathrm{i}}

\def\bZ{{\mathbb {Z}}}

\def\bR{{\mathbb {R}}}
\def\bC{{\mathbb {C}}}
\def\bP{{\mathbb {P}}}

\def\bH{{\mathbb {H}}}

\def\bT{{\mathbb {T}}}
\def\bS{{\mathbb {S}}}
\def\bI{{\mathbb {I}}}

\def\pA{{\mathcal A}}
\def\pB{{\mathcal B}}
\def\pC{{\mathcal C}}
\def\pD{{\mathcal D}}

\def\pF{{\mathcal F}}

\def\pO{{\mathcal O}}
\def\pP{{\mathcal P}}
\def\pQ{{\mathcal Q}}

\def\tnk{{\mathrm {TN}}_k}

\DeclareMathOperator{\diag}{diag}

\renewcommand{\ker}{\mathop{\mathrm{Ker}}\nolimits}

\makeatletter
\newcommand{\specialcell}[1]{\ifmeasuring@#1\else\omit$\displaystyle#1$\ignorespaces\fi}
\makeatother

\begin{document}
\title{Instantons and Bows for the Classical Groups}

\author{Sergey A. Cherkis\thanks{
Department of Mathematics, University of Arizona, 
617 N. Santa Rita Ave, 
Tucson, AZ 85721-0089, USA, 
\tt cherkis@math.arizona.edu}
\and
Jacques Hurtubise\thanks{
Department of Mathematics,
McGill University,
Burnside Hall,
805 Sherbrooke St.W.,
Montreal, Que. H3A 2K6,
Canada,
\tt jacques.hurtubise@mcgill.ca}}

\maketitle

\abstract{The construction of Atiyah, Drinfeld, Hitchin, and Manin  \cite{Atiyah78} provided complete description of all instantons on Euclidean four-space.  It was extended by Kronheimer and Nakajima to instantons on ALE spaces, resolutions of orbifolds $\mathbb{R}^4/\Gamma$ by a finite subgroup $\Gamma\subset SU(2).$  We consider a similar classification, in the holomorphic context, of instantons on some of the next spaces in the hierarchy, the ALF multi-Taub-NUT manifolds, showing how they tie in to the bow solutions to Nahm's equations \cite{Cherkis:2008ip} via the Nahm correspondence. Recently in \cite{Nakajima:2018bpn} and \cite{Nakajima:2016guo},  based on  \cite{MR2023313}, Nakajima and Takayama  constructed  the Coulomb branch of the moduli space of vacua of a quiver gauge theory, tying them to the same space of bow solutions. One can view our construction as describing the same manifold as the Higgs branch of the mirror gauge theory \cite{Cherkis:2011ee}. 
Our construction also yields the monad construction of holomorphic instanton bundles  on the multi-Taub-NUT space for  any classical compact Lie structure group.}

 
\section{Introduction} 

This paper is concerned with the study of instantons with classical structure groups on a multi-Taub-NUT manifold (described in Section \ref{Sec:mlds}). These play central role in the geometric Satake correspondence relating affine Grassmannian of a reductive group to the representation theory of its Langlands dual group  \cite{Braverman:2016pwk,Nakajima:2018ohd} and  in the physics view of the geometric Langlands duality for complex surfaces \cite{Witten10,Witten:2009mh}.   In string and M theory they govern the effective  dynamics of the Chalmers-Hanany-Witten brane configurations \cite{Cherkis:2008ip,Witten:2009xu}.  They also hold key to the numerous aspects of  supersymmetric quantum gauge theory:  the study of its 't Hooft operators \cite{Brennan:2018yuj}, supersymmetric boundary conditions \cite{Gaiotto:2008ak}, its spaces of vacua \cite{Cherkis:1997aa,Nakajima:2016guo}, and the gauge theory mirror symmetry \cite{Intriligator:1996ex,Aharony:1997bx,Cherkis:2011ee}.  There is a recent construction  \cite{Nakajima:2018bpn},  using the reflection functors of  \cite{MR2023313}, describing the moduli spaces of instantons with classical structure groups, is particular for instantons on Asymptotically Locally Euclidean (ALE) spaces.  An interesting physics construction of the moduli space of $SO(8)$ instantons on ALE space can be found in \cite{Tachikawa:2014qaa}.  Our focus is on a different base space, called multi-Taub-NUT, which is a prototypical Asymptotically Locally Flat (ALF)  space.   

 {Instantons on ALF spaces have been studied by various authors.  For instantons with trivial asymptotic holonomy, notably, Etesi and Hausel constructed one instanton on Taub-NUT \cite{Etesi:2001fb} and on multi-Taub-NUT \cite{Etesi:2002cc},  Etesi  \cite{Etesi:2011fy} established integrality of the $L^2$ curvature norm, 
Etesi and Szabo in \cite{Etesi:2008ew} found the moduli space of one instanton and described the corresponding holomorphic bundles on the twistor space, and 
Etesi and Jardim in \cite{Etesi:2006pw} use the Hausel--Hunsicker--Mazzeo compactification to find the dimension of the instanton moduli space.  All of these results assume that the instantons has trivial asymptotic holonomy; here we focus on the case of    instantons with generic asymptotic holonomy.}

The multi-Taub-NUT manifold $\tnk$ maps to $\bR^3$, with a circle as generic fibre. The boundary conditions for instantons on the Taub-NUT mimic those of the $\bR^3$ (Bogomolny) monopoles, and indeed a lot of the structures we consider have a strong filiation with the ones brought to light for monopoles. For example, there is a Nahm  correspondence, tying the instanton to a solution of the Nahm equations; the difference here being that we will be dealing with solutions on a circle (or, rather, a bow to be exact), exhibiting a peculiar sort of quasi-periodicity, rather than solutions on an interval.

The Taub-NUT manifolds are also hyperk\"ahler; as such they have a twistor transform, tying the instantons on the manifold to holomorphic objects on an associated twistor space, which is (differentiably) $\tnk\times \bP^1$. Furthermore, (via a well trodden, if in this case conjectural, path) there is  a Kobayashi-Hitchin correspondence, which says that the instantons are determined by their restriction as  holomorphic objects on the twistor space to a single twistor fibre over $ \bP^1$; this is equivalent to just keeping on $\tnk$ the $\overline\partial$ operator defined by the instanton corresponding to one fixed complex structure in the hyperk\"ahler family. Such twistor and Kobayashi-Hitchin correspondences hold  also on the Nahm side. Schematically, borrowing from our previous paper \cite{Cherkis:2017pop}, there is a diagram of bijective correspondences:

$$\begin{matrix} \begin{pmatrix} \text{Solns to}\\ \text{the ASD }\\ \ \text{equation}\\ \ \text{on}\ \tnk  \end{pmatrix} & \longleftrightarrow & 
\begin{pmatrix} \text{Holom.}\\  \text{vector bundles} \\ \ \text{on}\ \tnk\times \bP^1 \\ \ \text{+ \ conditions}\end{pmatrix}& \longleftrightarrow & 
\begin{pmatrix} \text{Holom. stable}\\ \ \text{vector bundles} \\ \ \text{on}\ \tnk \end{pmatrix}\\ \\
\quad \text{Up}\Big\uparrow\ \ \Big\downarrow\text{Down}&& \Big\updownarrow&&\Big\updownarrow\\ \\
\begin{pmatrix} \text{Bow Solns:}\\ \text{Solns to Nahm's Eqs}\\ \text{+\ Linear Data} \end{pmatrix} & \longleftrightarrow & \begin{pmatrix} {\rm \ Spectral }\\ \ {\rm data\  } \\ \ {\rm on\ }  T\bP^1 \\ \ {\rm + \ conditions}\end{pmatrix}& \longleftrightarrow & \begin{pmatrix} {\rm \ Holomorphic\  }\\ \ {\rm \ bow \ complex} \\ \ {\rm data\ on\ \bP^1}\end{pmatrix}
\end{matrix}$$

The difference here is that we deal with any multi-Taub-NUT space $\tnk$, not just the ordinary (first) Taub-NUT, and, significantly, with all classical groups as gauge groups.    {The  items in the left hand and middle columns encode in varying degrees of explicitness full solutions; the data on the righthand side   require solving a variational problem to  {recover} an instanton, and so do not tell us much about actual solutions; they are much easier to classify, and so inform us about moduli. More explicitly,  on the various components: 
\begin{itemize}
\item Along the top row, the correspondence between the left hand side and the middle is the twistor correspondence, saying that a connection is anti-self-dual iff it defines an integrable $\overline\partial$ operator for each of the complex structures of the hyperk\"ahler family; in this case to make the correspondence more explicit,  the analytical work of \cite{Cherkis:2016gmo} should give the necessary vanishing theorems.
\item The twistor construction considers all complex structures; the map from the middle  to the right hand side on the top just keeps one of them, and extends the bundle to a compactification; this is done here.
\item The correspondence between the left and the right along  the top row is the Kobayashi-Hitchin correspondence. The map from left to right is mostly forgetting structure, (just keeping the $(0,1)$ part of the connection) and compactifying; this is considered below. The correspondence from right to left is for this case conjectural, but there is a well established pattern of these results, following on the work of Donaldson \cite{Donaldson85}, Uhlenbeck-Yau \cite{UhlenbeckYau} and Simpson \cite{Simpson}, and several others.
\item Along the bottom row, the  left to middle correspondence is the well known linearization of Nahm's equations following Hitchin's original work \cite{Hitchin:1983ay}, as extended to multiple intervals in \cite{Hurtubise:1989qy}.
The bottom row, between left and right  is again the Kobayashi-Hitchin correspondence for Nahm's equations, showing that a solution to Nahm's equations with the appropriate boundary behaviour is encoded by its holomorphic part, a ``Nahm complex'', following the pattern established for Euclidean mono\-poles \cite{Donaldson:1985id}, \cite{Hurtubise:1989wh} . For the case at hand, this is  established now   by a recent paper of Nakajima and Takayama \cite{Nakajima:2016guo}, so that the bottom row is indeed a set of equivalences.
\item The top and bottom rows, on the left, are linked by the Nahm transform; in our current case, this is currently being studied in \cite{Second}.
\end{itemize}
We will focus mostly on the right-hand side of the diagram, top to bottom. While we will be establishing just the bijectivity of the maps, we note that the maps in question are all obviously continuous  and holomorphic . For a $U(n)$ gauge group, we will be establishing a bijection between }
 {
\begin{itemize} 
\item Holomorphic vector bundles $E$ (with extra structure) on a compactification $X$ of the multi-Taub-NUT $\tnk$, as described in Theorem \ref{instanton-to-bundle}, arising from an instanton on the Taub-NUT;
\item Sheaves $Q^i_{n-i}, i=0,..,n, P^i_{n-i-1}, i=0,...,n-1,$ on the Riemann sphere, (defining in addition an auxiliary sheaf $R$ on $X$), and maps between the sheaves, as explained in Propositions \ref{degreeofQ} and \ref{P,Q-properties}.
\end{itemize}}

 {The structures of the sheaves $P, Q$, and the maps between them, encode in a fairly natural way the {\it Nahm complex},    the holomorphic part of the solution to Nahm's equations, with appropriate boundary conditions;  in our case these solutions live  on a sequence of intervals, arranged along a circle; this is the subject of Sections \ref{Sec:Bows} and \ref{Sec:BowClxs}. In our context, the appropriate solutions to the Nahm equations  {(with the boundary linear data)} are referred to as bow solutions, and the associated  Nahm complex is referred to as a bow complex.  The  work of  \cite{Nakajima:2016guo}   turns these bow complexes into  unique  bow solutions, as expounded in  \cite{Cherkis:2009jm}, and so establishes the Hitchin-Kobayashi correspondence in this particular case.}
 
  { Taking resolutions of the sheaves $P, Q$ and of the maps between them also gives us a family of matrices from which one can construct a monad description of the bundle $E$ and its auxiliary structure.}  {This} theme of reducing the study of instantons and monopoles, in particular their moduli,  to algebraic data, and to solutions of ordinary differential equations, has quite a history, with the original work of  Nahm, Hitchin, and Donaldson on monopoles \cite{Nahm:1979yw, Nahm84, Hitchin:1982gh, Hitchin:1983ay, Donaldson:1985id}; further work of Murray and Hurtubise-Murray extending this to the classical groups 
\cite{Murray:1983mn, Hurtubise:1989qy, Hurtubise:1989wh};   then of Jarvis for arbitrary compact groups  \cite{Jarvis1,Jarvis2}; the beautiful extension of these ideas by Kronheimer and Nakajima to instantons on Asymptotically Locally Euclidean (ALE) spaces   \cite{KN}; the passage from monopoles to calorons using the monad description of  Charbonneau and Hurtubise \cite{Charbonneau:2006gu}; and the bow construction of $U(n)$ instantons on Asymptotically Locally Flat (ALF) spaces  \cite{Cherkis:2009jm,Cherkis:2010bn} and its monad description \cite{Cherkis:2017pop}. This is but a partial list.  Many of the techniques of these earlier papers reappear here.  

All this work, however, has its roots in the pioneering paper of Atiyah, Hitchin, Drinfeld and Manin \cite{Atiyah78}. It is therefore both an honour and a pleasure to have our work appear in a volume in memory of Michael Atiyah, whose mathematics, in its constant search for deep and interesting connections, and its simultaneous aesthetic sense, is a model for us all.

\section{The Taub-NUT manifolds}

\subsection{The manifolds}\label{Sec:mlds}
The multi-Taub-NUT manifolds $\tnk$ are the basic examples of ALF (Asymptotically Locally Flat) hyperk\"ahler four-manifolds. 
They are constructed by choosing an $S^1$ bundle over $\bR^3$ minus $k$ points $p_1, ...,p_k$, which is isomorphic to the  Hopf (tautological) bundle $\bH= S^3\rightarrow S^2$ on spheres surrounding  each of the points $p_i$; over a sphere surrounding all of the points, it is then isomorphic to the $k^{th}$ power $\bH^{k}$ of the tautological  bundle.  The multi-Taub-NUT  manifold $\tnk$ is obtained from this  $S^1$-bundle by glueing in single points $q_1,...,q_k$ over each of the $p_1, ...,p_k$. The local model for this is the Hopf map $B^4\rightarrow B^3$.  The result is a four-manifold
$$\pi: \tnk\rightarrow \bR^3.$$ There is a commutative diagram
\begin{equation}\label{Hopf}\begin{matrix} \bH^k&\rightarrow &\tnk\setminus\{q_1,...,q_k\}& \rightarrow &\tnk\\
 \downarrow&&\downarrow&&\downarrow\\ S^2 &\rightarrow &\bR^3 \setminus \{p_1, ...,p_k\}&\rightarrow &\bR^3.\end{matrix}\end{equation}
Here the $S^2$ is a sphere out near infinity in $\bR^3$.
We have the following lemma, giving the basic topological invariants of $\tnk$ and of $\bH^k$:

\begin{lemma}
\begin{itemize}

\item $\pi_1(\tnk) = H_1(\tnk,\bZ) = H^1(\tnk,\bZ) = 0$;
\item $H_2(\tnk,\bZ)= H^2(\tnk,\bZ)=\bZ^{k-1}$; a natural basis for $H_2(\tnk,\bZ)$ is given by the inverse images $\gamma_i$ of segments joining $p_{i}$ to $p_{i+1}$, $i=1,...,k-1$. 
 \item $\pi_1(\bH^k) = H_1(\bH^k,\bZ) = \bZ/k\bZ$, $H^1(\bH^k,\bZ) =0$;
\item $H_2(\bH^k,\bZ) = 0 ;  H^2(\bH^k,\bZ) =  \bZ/k\bZ$;
\item The diagram (\ref{Hopf}) gives, on the level of $H^2(\cdot, \bZ)$:
$$\begin{matrix}  \bZ/k\bZ&{\buildrel{(1,...,1)}\over{\leftarrow}} &\bZ^{k-1}& = &\bZ^{k-1}\\
 \uparrow&&\uparrow M&&\uparrow\\\bZ &{\buildrel{(1,...,1)}\over{\leftarrow}}&\bZ^k&\leftarrow &0\end{matrix}$$
 In the dual of the basis given above for $H_2(\tnk,\bZ)$ and   the natural basis for $H^2(\bR^3 - \{p_1, ...,p_k\},\bZ)$ given by the duals of small spheres around the $p_i$, choosing suitable orientations,  the map $M$ is given by the matrix
 $$\begin{pmatrix} -1&1&0&...&0&0\\ 0&-1&1&...&0&0\\ ..&..&..&&&..\\0&0&0&...&-1&1\end{pmatrix}$$
\end{itemize}
\end{lemma}
The proof uses the fact that one has a covering space $\bH\rightarrow \bH^k$, the universal coefficients theorem relating homology and cohomology and the Mayer-Vietoris sequences giving the construction of $\tnk$ inductively from $\bR^4 =  {\mathrm {TN}}_1$ and ${\mathrm {TN}}_{k-1}$.

To get a hyperk\"ahler manifold, one must of course give the metric. This is done via the Gibbons-Hawking ansatz, which will give us a manifold with an $S^1$-triholomorphic action. The orbits of this action tend asymptotically to circles of a constant length: explicitly there are local coordinates $(t_1,t_2, t_3,\theta)\in \bR^3\times [0,2\pi)$ in which the action of $S^1$  by a linear shift in $\theta$ is isometric,  and the metric  locally has the Gibbons-Hawking form \cite{Gibbons:1979zt}:
\begin{align}\label{TNmetric}
ds^2 = V(t) d\vec{t}\phantom{|}^{2} + \frac{(d\theta+ \omega \,)^2} {V(t)},
\end{align}
with  
$$V(t) = \ell + \sum_{i=1}^k\frac{1} { 2|\vec{t}-\vec{p}_i |},$$
where $\ell>0$ is a fixed parameter determining the asymptotic size of the $S^1$
and the local one-form $\vec{\omega}\cdot d\vec{t}$ appearing in the metric is related to $V$ by 
$\frac{\partial}{\partial t_i}  V = \epsilon_{ijk}\frac{\partial}{\partial t_j}\omega_k.$ 
The local $S^1$ fibre coordinate $\theta$ is identified with  $ \theta+2\pi$ and it is a local coordinate in a chart $\pi^{-1}(U)$ over a contractible region $U$ in  $\mathbb{R}^3\setminus \{p_1, ...,p_k\}$.  Note, that although $\theta$ and $\omega = \vec{\omega}\cdot d\vec{t}$ are local, the one-form $d\theta+\vec{\omega}\cdot d\vec{t}$ is a global one-form dual to the isometry generating vector field $\frac{\partial}{\partial\theta}.$

\subsection{Twistor spaces} 
The Taub-NUT manifolds are hyperk\"ahler, and so have a Riemann sphere's worth of complex structures; these get encoded in a twistor space which is, as a differentiable manifold, $\mathcal{Z}=\tnk\times\bP^1$.  It is a holomorphic fibration $\mathcal{Z}\rightarrow\mathbb{P}^1$ with fibre $\tnk(\zeta)$, the manifold equipped with the complex structure $I_\zeta$,  over $\zeta\in\mathbb{P}^1.$ This space $\mathcal{Z}$ lies above the complex surface $\bS$ that is the total space  of $\rho: \pO(2)\rightarrow \bP^1$; $\bS$ is called the mini-twistor space of $\mathbb{R}^3$; its points are oriented lines in $\bR^3$, with coordinates $(\eta,\zeta)$ corresponding to  $\eta\partial/\partial\zeta\in\pO(2)=t\bP^1.$ 

Referring to Hitchin's paper \cite{Hitchin:1982gh} the twistor correspondence gives for a point $(t_1,t_2,t_3)$ in $\bR^3$, the complex section of the fibering $\rho: \bS=\pO(2)\rightarrow \bP^1$
$$\eta=x(\zeta) := (t_1+\ii t_2) -2t_3\zeta + (-t_1+\ii t_2)\zeta^2.$$
where $\zeta$ is a coordinate in $\bP^1$  and $\eta$ the fibre coordinate.For example, $\eta=p_i(\zeta)$  is a section of $\mathcal{O}(2)$  corresponding to $\vec{p}_i\in\mathbb{R}^3.$ The surface  $\bS ={\mathrm{Tot}\, } \pO(2)$ has lying above it   line bundles $L^\mu$, whose transition functions from $\zeta\neq 0$ to $\zeta\neq \infty$ are given by $\exp(\mu\eta/\zeta)$. We set $  L^\mu(j) :=   L^\mu\otimes \pO(j) $, 
and take over  $\bS$ the subbundle $\bT$ of conics lying inside the rank two bundle $L^\alpha(k)\oplus L^{-\alpha}(k)$ given over $\zeta\neq \infty$ by 
$$\bT = \{ (\xi,\psi)\in (L^\alpha(k)\oplus L^{-\alpha}(k))|_{(\eta, \zeta)}|\ \xi\psi = \prod_{i=1}^k (\eta-{p_i}(\zeta))\}.$$
Over the line $\mathfrak{l}$ in $\bR^3$ corresponding to $(\eta, \zeta)\in\bS$, the conic is either a cylinder (if all $p_i$ are disjoint from $\mathfrak{l}$) or the union of two disks at a point (if one $p_i\in \mathfrak{l}$) or a chain of a disk, a chain of $\mathbb{P}^1$'s, and another disk, if several of the $p_i\in \mathfrak{l}$. In all  cases, the two end annuli are metrically asymptotic to cylinders of radius $2\pi/\sqrt{\ell}.$

Fixing a complex structure (for convenience, say $\zeta= 0$) there is a complex surface projecting to the $\eta$ line:
$$\tnk (0) = \{ (\xi,\psi,\eta )\in \bC^3|\ \xi\psi = \prod_{i=1}^k (\eta-p_i(0))\}.$$
Without  loss of generality, we presume all $z_i =  {p_i}(0)=p_j^1+\ii p_j^2$ distinct; otherwise, we just shift $\zeta = 0$ to another value.

\subsection{ Compactifying the conics}

We now build a holomorphic  compactification $X$ of $\tnk(0)$ by 
\begin{itemize}
\item Adding a point to each end of each conic over the plane $\eta \in \bC$; this turns our cylinders over $\eta\neq  z_i$ into Riemann spheres  or, over $\eta=  z_i$, into chains of two Riemann spheres; call this space $X_0,$ 
\item Adding an extra $\bP^1$ over $\eta = \infty$, compactifying  $X_0$ to $X$.
\end{itemize}
The space $X_0$ is obtained from $\bP^1\times \bC$ by blowing up $k$ points $(\psi , \eta) = (0,  z_i)$;  
$X$ in turn, can be obtained from $\bP(\pO(k)\oplus \pO)$  over $\bP^1$ by blowing up $k$ points $((\psi,\phi),\eta)= ((0,1), z_i)$. Set $\xi = \phi \cdot  \prod_{i=1}^k (\eta- z_i)$.

The resulting $\pi:X\rightarrow \bP^1$  then has 
\begin{itemize}
\item An ``infinity divisor'' $C_\infty$, the proper transform of the section $(\psi,\phi)=(0,1)$; this has self-intersection $0$.
\item A ``zero  divisor'' $C_0$ given by $\phi = 0$; this has self-intersection $-k$.
\item A generic fibre $F$ of the original ruled surface, with self intersection 0.
\item Exceptional   fibres $D_i = D_{i, \xi}\cup D_{i,\psi}$ given as the inverse image  $\pi^{-1}(z_i)$ in $X$ of $z_i$ under $\pi$. Both   $D_{i, \xi}$ and $D_{i,\psi}$ have self-intersection $-1$. One has $D_{i, \xi}\cap C_ 0 = 1,D_{i, \xi}\cap C_ \infty = 0,  D_{i, \psi}\cap C_ 0 = 0,D_{i, \psi}\cap C_ \infty = 1$.
\end{itemize}
 The complex curves $F,  D_{i,\psi},  D_{i,\xi}, C_0, C_\infty$ are all   projective lines.
We have that $\psi = 0$ on $C_\infty \cup (\cup_{i=1}^k D_{i,\psi})$, and $\psi = \infty$ on $C_0$;  the function $\xi$, in turn, is zero on $C_0 \cup (\cup_{i=1}^k D_{i,\xi})$, and infinite on $C_\infty$.  
\begin{figure}[htb]
\begin{center}
    \includegraphics[width=0.35\textwidth]{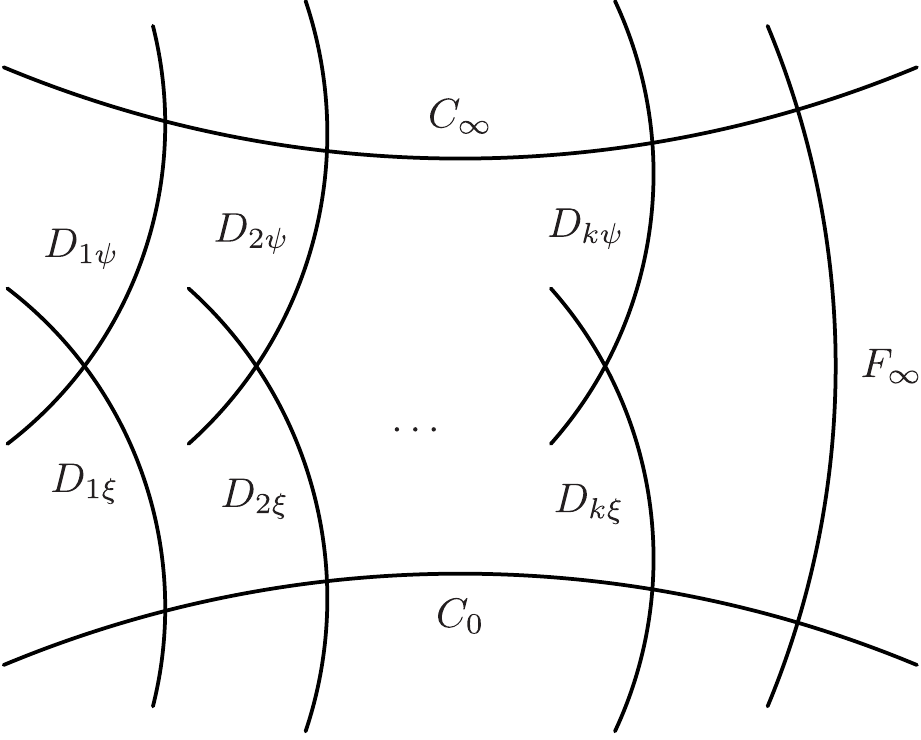}
\label{Inter_k}
\end{center}
\end{figure}

\begin{lemma}\label{Lem:TopX}
\begin{itemize}

\item $\pi_1(X) = H_1(X,\bZ) = H^1(X,\bZ) = 0$;
\item $H_2(X,\bZ)= H^2(X,\bZ)=\bZ^{k+2}$; a  basis for $H_2(X,\bZ)$ is given by the curves $C_\infty, F, D_{i,\xi}, i=1,...,k$.
\item The inclusion $\tnk\rightarrow X$ induces on second  homology a map, given in the bases   $\gamma_i$ for $H^2(\tnk,\bZ)$, and   $C_0, F, D_{i,\xi}, i=1,...,k$ for $H^2(X,\bZ)$, by the matrix
$$\begin{pmatrix} 0&0& 0&...&0&0\\0&0& 0&...&0&0\\-1&0&0&...&0&0\\1&-1&0&...&0&0\\0&1&-1&...&0&0\\0&0&1&....&0&0\\ ...&...&...& &...&..\\0&0&0&...&1&-1\\0&0&0&...&0&1\\\end{pmatrix}$$

\end{itemize}
\end{lemma}

\begin{remark} 
A very useful compactification of the multi-Taub-NUT as a smooth differential manifold is given by Hausel, Hunsicker, and Mazzeo in \cite{Hausel:2002xg}.  This compactification adds one point at infinity for each direction in the base $\mathbb{R}^3.$  In particular, the HHM compactification of the Taub-NUT, TN$_1$, is $\mathbb{P}^2.$  {This compactification is the one used in \cite{Etesi:2006pw} to compute the dimension of the instanton moduli space.}
Unfortunately, this compactification is not really compatible with  the complex structures of the Taub-NUT. For this reason we employ the compactification $X$ specified above. 
\end{remark}

\section {Instantons on $\tnk$}

Let $E_0$ be a $U(n)$-bundle over $\tnk$ with  a first Chern class $c_1(E_0)=\sum_{i=1}^{k-1} r_i \gamma_i^*$ in $H^2(\tnk,\bZ)$, where $\gamma^*_i$ are the classes dual to the $\gamma_i. $
We consider  $U(n)$ instantons on $E_0$, that is connections $\nabla$ on $E_0$ with curvature having finite $L^2$ norm and generic asymptotic holonomy, as defined in \cite{Cherkis:2016gmo}, i.e. there exists a ray in the base $\mathbb{R}^3$ such that there is a limit, going  to infinity over the ray, of the  holonomy of the instanton connection around the $S^1$ fibres of $\tnk\rightarrow\bR^3$, along the ray, and that this limit is regular.  {For $U(n)$ structure group} these conditions imply \cite[Thm. 22]{Cherkis:2016gmo}
that there exist local frames such that the connection one form has the  form
$$A = -\ii\diag \left((\lambda_j +\frac{m_j}{2r}) \frac{d\theta + \omega}{V} - \frac{m_j}{k}\omega\right) + O(\frac{1}{r^2}),$$
where $\lambda_j$ are real, $\exp(2\pi \ii\lambda_j/\ell)$ are distinct, and the $m_j$ are integers called {\em monopole charges}. There are also good asymptotic curvature bounds, of order $1/r^2$.
  
\subsection{Abelian instantons}

 The Taub-NUT has a   continuous family of basic  $U(1)$ instantons with a globally defined connection one-form
$$a_s = \ii s \frac{d\theta +\omega}{V},$$
for $s\in \bR$.
 These instantons are on line bundles which are trivial. More generally, one has connection one-forms  {locally} of the form 
\begin{equation}\label{Abelianinstantons} a_H= \ii (H \frac{d\theta +\omega}{V} -  \eta_H),\end{equation}
with $d \eta_H=*dH$ and $H$ harmonic on $\bR^3$,  with simple poles at the $\vec{p}_i$, and integer/2 polar parts at these points:
 $$ H = s + \sum_{i=1}^k \frac {n_i}{2|x-p_i|}$$
 These are essentially Dirac monopoles \cite{Kronheimer}; they are defined over bundles $L_s=L^s\otimes\otimes_{\sigma}{R_\sigma}^{\otimes n_\sigma}$ of degree $n_{i+1}-n_i$ on the cycles $\gamma_i$. Expanding around infinity, one has 
 $$\frac{H}{V} = (s+ \frac {\sum_{i=1}^k   n_i}{2r})(\ell^{-1}(1-\frac{k}{2r\ell}))+ O(r^{-2})= \frac{s}{\ell} + \frac{-\frac{sk}{\ell }+\sum_{i=1}^k   n_i } {2\ell r}+ O(r^{-2}),$$
 and so the $d\theta$ term of the connection matrix is
 $$a_H(\partial/\partial\theta)=\ii( \frac{s}{\ell} + \frac{-\frac{sk}{\ell }+\sum_{i=1}^k   n_i } {2\ell r})+ O(r^{-2}).$$
 
In this   trivialization, the connection has limiting $d\theta$ term $\ii \frac{s}{\ell}$. (Local) changes of trivialization of the form $e^{\ii n_0\theta}$ near infinity transform the connection, giving integer shifts in the leading term $ \frac{s}{\ell}$ of the $d\theta$ portion of the connection, while the subleading $r^{-1}$ term is an invariant; this amounts to twisting by the (trivial) lift of the line bundle $\pO(n_0)$ from the two-sphere. 
Note then that even after this transformation 
$$ k({\mathrm{ leading\ term}}) +2\ell (r^{-1}{\mathrm {term}})=k n_0+\sum_{i=1}^k n_i $$ is an integer, a magnetic charge $m = kn_0 +\sum_{i=1}^k   n_i $. It is not an invariant, but depends on the trivialization, and so does the normalisation of the leading term.  One can, for example, normalise $ \frac{s}{\ell}$ to lie in $[0,1)$, which then fixes the magnetic charge. See  \cite{Cherkis:2016gmo}.

\subsection {$U(n)$ instantons} 
The generic asymptotic holonomy assumption implies an asymptotic reduction of the $U(n)$ gauge group to its maximal torus $U(1)^n,$ given as the stabilizer of the holonomy along the $\theta$-direction. Changes in trivialisation of the form $\exp(\diag(n_i)\theta)$ shift the asymptotic limits of the connection in $\theta$ (logarithms of the holonomy) $\lambda_j/\ell$ by integers; this also modifies the $1/r$ terms of the connections, as for the Abelian case. We normalise so that 
\begin{equation} \label{ordering} 0\leq \lambda_1/\ell<\lambda_2/\ell<...<\lambda_n/\ell<1\end{equation}
With this, the monopole charges $m_j$ become fixed.

{The anti-selfduality equations give us a holomorphic vector bundle over $\tnk$. We would like to extend this to a bundle over $X$. This involves extending first along the lines $\eta =$ constant, and then moving out to $\eta = \infty$. }

{For extending from a  family of  cylinders  to a  family of spheres, let us briefly consider a model case,  a line bundle with a unitary connection $\ii \frac{s}{\ell} d\theta$ along a cylinder coordinatized by $\xi = re^{\ii \theta}$. Changing trivializations by  $(\xi\overline\xi)^{s}$ kills the  {$(0,1)$} component of the connection; we are then in a holomorphic gauge, which we can then extend to $\xi = 0$;  the norm of the holomorphic section decays as $(\xi\overline\xi)^{s}$. This holomorphic extension applies more generally to our situation:}

\begin{theorem}\label{instanton-to-bundle}
Let $(E_0,\nabla)$ be a $U(n)$ instanton on $\tnk$, with asymptotic eigenvalues $\lambda_i/\ell$ of the logarithm of the $\theta$-holonomy normalized in a gauge at infinity as in (\ref{ordering}): $0\leq \lambda_1/\ell<\lambda_2/\ell<...<\lambda_n/\ell<1$,   with magnetic charges $m_j$, and with $c_1(E_0) = \sum_ir_i\gamma_i^*$.
Then $E_0$ extends to a holomorphic vector bundle $E$ over $X$, where
\begin{itemize} 
\item 
$E$ is trivial on the generic fibre of $\pi:X\rightarrow \bP^1$, in particular the fibre of $\eta=\infty$.
\item 
$E|_{C_0}$ is filtered by a full flag of subbundles $E_{0,1}\subset E_{0,2}\subset ...E_{0,n}= E$, with $\deg{E_{0,i}} = m_1+ ... +m_i,$ 
\item 
$E|_{C_\infty}$ is filtered by a full flag of subbundles $E ^{\infty,1}\subset E ^{\infty,2}\subset ...E^{\infty,n} = E$, with $E^{\infty,i}$ trivial. 

\item 
The first Chern class of $E$ is $\sum n_i D_{i,\xi}$, with 
$$m= \sum_{j=1}^n m_j=\sum_{j=1}^k n_j, \ n_{j+1} - n_j =  r_j, j = 1,..,k-1.$$
\item 
The second Chern class $m_0= m_0(E)$ is determined by a curvature $L^2$ norm, as in  \cite{Cherkis:2016gmo}.
 \end{itemize}
\end{theorem}
\begin{proof}  { We first extend to $\xi= 0, \psi = 0$, away from $\eta = \infty$. This follows in a straightforward way the proof given by Biquard  \cite{Biquard97}, in his work on parabolic bundles, using the decay estimates of  \cite{Cherkis:2016gmo}. One can also see \cite{Charbonneau19}, for a case where the base is asymptotically a cylinder, similarly to here, and in which the holomorphic extension is discussed at length. The argument basically shows that our passage to holomorphic basis given in our model problem holds more generally.   Along each eigenspace of the angular portion of the connection, the passage to a holomorphic trivialisation gives an asymptotic growth (or decay) of $(\xi\overline\xi)^{\lambda_j/\ell}$, where $\lambda_j$ is the eigenvalue. Thus different eigenvalues give different rates, which are  incommensurable, and there is a natural flag on the extension given by (maximal decay rate) $\subset$ (next to maximal decay rate) $\subset...$, as is standard for parabolic bundles. Note that the two extensions at $\xi = 0$ and $\psi = 0$ (essentially,  $\psi = 0$ corresponds to $\xi=\infty$) flip the order of the eigenspaces, and so of the flag; we get  a flag $E_{0,j}$ along $C_0$ corresponding to holomorphic sections of norm $<C|\xi\bar{\xi}|^\frac{\lambda_j}{\ell}, $  and a flag $E^{\infty,j}$ corresponding to holomorphic sections of norm $<C|\psi\bar{\psi}|^{-\frac{\lambda_{n-j}}{\ell}}.$ }

 { Now as one goes out to infinity in $\eta$, one has a bound (up to a constant) of $1/ (|a|^2+ (|\xi|^2 +|\psi|^2))$ for the curvature along $\eta= a$, by the work of \cite{Cherkis:2016gmo}. For the bundle to be non-trivial, on a line,  one would need a positive sub-bundle on that line, and which must have enough curvature for a positive Chern class; since curvature on holomorphic  subbundles is bounded  above by curvature on the whole bundle, this means that the integral of the norm of the curvature on the bundle along the line should be large. On the other hand, the integral of the curvature is bounded by a constant times $1/|a|$, meaning that for $a$ sufficiently large any holomorphic subbundles must be of zero or negative degree; the bundle is then trivial on the line.   In a similar fashion, the flags at both ends of the cylinders have to be transversal in a global trivialization, as the asymptotic trivialisation becomes asymptotically flat along the line $\eta= a, a\rightarrow \infty$.}
 
 {Now trivialize the bundle and the flag along $C_\infty$; this also trivializes the bundle on the set $|\eta|>R$. To this, one can   then glue in a trivial bundle (and flag) on the product of a disk times $\bP^1$ to extend the bundle to $X$, with, in particular, the bundle being trivial over $\eta =\infty$.}

 {For the first Chern class $c_1(E)$, the group $H^2(X,\bZ)$, by Lemma~\ref{Lem:TopX}, is generated by a fibre $F_\infty$, the curve $C_\infty$ and the exceptional divisors $D_{j,\xi}$. Write $c_1(E)$ as $n_F F_\infty + n_\infty C_\infty + \sum n_i D_{i,\xi}$; the first Chern class of the restriction of the bundle to a divisor $D$ is given by the intersection of $c_1(E)$ with $D$.
By construction, our bundle is trivial on $C_\infty$, and on the generic fibre $F$; this forces $n_F= n_\infty = 0$. In the compactified $X$, $\gamma_j = D_{j+1,\xi}-D_{j,\xi}$. This gives $\gamma_i^*= D_{1,\xi}+ D_{2,\xi}+...+D_{i,\xi}$, up to a multiple of $\sum_iD_{i,\xi}$ 
On the other hand, we already had that as a bundle on $\tnk$, $c_1(E_0) = \sum_{j=1}^{k-1} r_j\gamma_j^*$;   
(implying
$n_j = \sum_{i=j}^{k-1}r_i,  j = 1,..,k-1$, again modulo $\sum_iD_{i,\xi}$, and so $n_{j}-n_{j+1}= r_j$ ) 
One can add the same integer to each of the $n_i$, and still have the same first Chern class of $E_0$. On the other hand,
the restriction of $E$ to the divisor $C_0$ has first Chern class 
$$m= \sum_{j=1}^n m_j=\sum_{j=1}^k n_j.$$ 
The second Chern class is computed  in \cite{Cherkis:2016gmo}.}
\end{proof}

  { In the Abelian case, one can be a bit more specific on the extension of the instantons (\ref{Abelianinstantons}): 
  \begin{proposition}
 The  instantons  \eqref{Abelianinstantons} yield holomorphic line bundles 
 $\pO(\sum_i (n_i +n_0) D_{i, \xi})$ over $X$.    \end{proposition}
\begin{proof} 
Indeed, here we get a line bundle, trivial on the generic fibre $F$, and also on $C_\infty$; on the other hand, it has degree $n_{i+1}- n_i$ on the cycles $\gamma_i = D_{i+1, \xi }-D_{i ,\xi}$, giving $\pO(\sum_i (n_i +n_0) D_{i, \xi})$ for some $n_0$.\end{proof}}

\subsection{Grothendieck-Riemann-Roch calculations, and some direct image sheaves}

We compute some characteristic classes for our bundle $E$, some related sheaves $E^i_j$, as well as some direct image sheaves $Q^i_{n-1}, P^i_{n-i-1}$, defined below.

We have for the characteristic classes of the tangent bundle of $X$
$$c_1(X) = (-k+2) F _\infty+2 C_\infty + \sum_i D_{i,\psi},\qquad c_2(X) = (4+k) pt,$$
where $pt$ is the point class in homology, dual to the top class in cohomology.
This gives for the Todd class
$$Td(X) = 1 + \frac{c_1}{2} + \frac{ c_1^2+c_2}{12} = 1 + \frac{ (-k+2) F_\infty +2 C_\infty + \sum_i D_{i,\psi}}{2} + pt.$$
The Chern character of $E$ in turn is 
$$ ch(E) = n+ c_1(E) + \frac{c_1(E)^2 - 2c_2(E)}{2} = n+ \sum_i n_i D_{i,\xi} + \frac{-2m_0-\sum_in_i^2}{2}pt.$$
Therefore 
\begin{align*} 
ch(E(-C_0)) &=  ch(E)\cdot ch(\pO(-C_0)) = ch(E)\cdot (1-C_0 -\frac{k}{2}pt)\\
&= n + [\sum_in_iD_{i,\xi}-nC_0] + [\frac{-\sum_in_i^2}{2} - m_0-\sum_in_i -\frac{nk}{2}]pt.
\end{align*}
and so 
\begin{multline*} 
ch(E(-C_0))Td(X) = n + [\sum_in_iD_{i,\xi}-nC_0 +n(\frac{ (-k+2) F_\infty +2 C_\infty + \sum_i D_{i,\psi}}{2})]\\ 
+ [\frac{-\sum_in_i(n_i+1)}{2} - m_0] pt.
\end{multline*}
Taking the integral along the fibre $\pi_*$ to $\bP^1$, one has for the Chern character of $\pi_!(E(-C_0))$:
$$0 + [\frac{-\sum_in_i(n_i+1)}{2} - m_0]pt.$$
 {\begin{definition}
Set 
$$Q_0^n {\buildrel{\mathrm def}\over {=}} R^1\pi_*(E(-C_0))$$
\end{definition}
Since $E$ is trivial on the generic line, the  zeroeth direct image of $E(-C_0)$ vanishes; the first direct image $Q_0^n$ is then a torsion sheaf 
 with  length $\frac{ \sum_in_i(n_i+1)}{2} + m_0.$ }
 
 {\begin{definition} \label{Eij} The subsheaves $E_i^{j}$ of $E$ are defined by the exact sequence
$$0\rightarrow E_i^{j} \rightarrow E\rightarrow  (E|_{C_0}/E_{0,i}) \oplus (E|_{C_\infty}/ E^{\infty,{j}})\rightarrow 0.$$
These are the subsheaves (they are locally free, hence bundles) of sections of $E$ taking values in the $i$-th component of the flag over $C_0$ and the $j$-th component of the flag over $C_\infty$. The indices $i,j$ can be zero; taking values in the $0$-th component of the flag simply means that the  section vanishes at the point. \end{definition}}

We begin by looking at $E_i^{n-i}$. The sheaves $ (E|_{C_0}/E_{0,i}),(E|_{C_\infty}/ E^{\infty,{n-i}})$ are supported over $C_0, C_\infty$ respectively; $(E|_{C_0}/E_{0,i})$, as a smooth vector bundle over $C_0$ is a sum of bundles $\sum_{j=i+1}^n\pO(m_j)$; $(E|_{C_\infty}/ E^{\infty,{n-i}})$ in turn is trivial over $C_\infty$. From the sequence 
$$0\rightarrow \pO(mD_{1,\xi} - C_0) \rightarrow \pO(mD_{1,\xi}) \rightarrow \pO(m)|_{C_0}\rightarrow 0,$$
one finds $ch(\pO(m)|_{C_0}) = 0 +C_0 + (m+k/2) pt$; similarly $ch(\pO|_{C_\infty}) = 0 + C_\infty$. This then gives 
\begin{align*}ch(E_i^{n-i}) &= ch(E) - \sum_{\ell=i+1}^n (C_0 + (m_\ell+k/2)pt) - \sum_{j=1}^{i} (C_\infty) \\
&= n + (\sum_j n_j D_{j,\xi} - (n-i)C_0 -i C_\infty) \\
&\qquad + \left(\frac{-\sum_jn_j^2}{2} - m_0-\sum_{\ell=i+1}^n m_\ell- \frac{(n-i)k}{2}\right)pt,\\
\end{align*}
and so 
\begin{multline*}
ch(E_i^{n-i})Td(X) = n + \Bigg(\sum_jn_jD_{j\xi} - (n-i)C_0 -i C_\infty\\
\qquad \qquad \qquad  +n(\frac{ (-k+2) F_\infty +2 C_\infty + \sum_i D_{i,\psi}}{2} )\Bigg) \\
 + \left(\frac{- \sum_j n_j( n_j-1 )}{2} - m_0-\sum_{\ell=i+1}^n m_\ell\right)pt.
\end{multline*}
Projecting to $\bP^1$, this give   $ch(\pi_!(E_i^{n-i}))Td(\bP^1) = ch(\pi_!(E_i^{n-i}))(1+pt)$:
$$0 +  \big(\frac{ -\sum_j n_j( n_j-1 )}{2} - m_0-\sum_{\ell=i+1}^n m_\ell\big)pt = 0 +  \big(\frac{- \sum_j n_j( n_j+1 )}{2} - m_0+\sum_{\ell=1}^i m_\ell\big)pt,$$
recalling that $\sum_in_i = \sum_\ell m_\ell$. 
\begin{definition} \label{defQ} Set 

$$Q_i^{n-i}\ {\buildrel{\mathrm def}\over {=}}\  R^1\pi_*((E_i^{n-i}))$$ \end{definition}

Since the zeroth direct image vanishes,   the first direct image $Q_i^{n-i}$ is a torsion sheaf of length 
$$   \frac{  \sum_j n_j( n_j-1 )}{2} + m_0+\sum_{\ell=i+1}^n m_\ell  = \frac{ \sum_j n_j( n_j+1 )}{2} + m_0-\sum_{\ell=1}^i m_\ell.$$
Note that we have two limiting cases, the first being $i=0$, where we are effectively computing the Chern character of $\pi_!(E(-C_0))$ and so the length of $Q_0^n$ , and the other being $i= n$, where we are computing the Chern character of $\pi_!(E(-C_\infty))$, and so the length of  
\begin{definition}
$$Q_n^0\ {\buildrel{\mathrm def}\over {=}}\  R^1\pi_*(E(-C_\infty)),$$\end{definition}
obtaining 
$$ \frac{  \sum_j n_j( n_j-1 )}{2} + m_0.$$
 To reflect these limiting cases, we set 
\begin{equation}\label{E0n} E_0^n = E(-C_0),\qquad E_n^0  = E(-C_\infty).\end{equation}
 Note that the difference in lengths between $Q_0^n$ and $Q_n^0$ is $\sum_in_i = \sum_\ell m_\ell $.
 
 Now compute for $E_i^{n-i-1}$; the same computation yields 
 \begin{multline*}
 ch(E_i^{n-i-1})Td(X) = n + (\frac{- \sum_j n_j( n_j-1 )}{2} - m_0-(\sum_{\ell=i+1}^n m_\ell) -1) pt\\
 + (\sum_jn_jD_{j\xi} - (n-i)C_0 -(i+1) C_\infty+n(\frac{ (-k+2) F +2 C_\infty + \sum_i D_{i,\psi}}{2} )).
\end{multline*}
Again project to $\bP^1$; one has  the  Chern character of $\pi_!(E_i^{n-i-1})$
 \begin{align*}
 &\left\{(-1) +  (\frac{ -\sum_j n_j( n_j-1 )}{2} - m_0- \sum_{\ell=i+1}^n m_\ell -1) pt\right\} (1-pt)\\
 &= (-1) +  (\frac{ -\sum_j n_j( n_j-1 )}{2} - m_0-\sum_{\ell=i+1}^n m_\ell) pt  \\ 
 &=  (-1)  +  (\frac{- \sum_j n_j( n_j+1 )}{2} - m_0+\sum_{\ell=1}^i m_\ell ) pt.   \end{align*}
Similarly, define
\begin{definition}\label{defP} 
$$P_i^{n-i-1}\ {\buildrel{\mathrm def}\over {=}}\  R^1\pi_*(E_i^{n-i-1})$$ \end{definition}
Again the zeroeth direct image vanishes, and  so the first direct image $P_i^{n-i-1}$ is a rank one sheaf of degree
$$\frac{  \sum_j n_j( n_j+1 )}{2} + m_0 - \sum_{\ell=1}^i m_\ell. $$
Summarising:
\begin{proposition}\label{degreeofQ} On $\bP^1(\bC)$:

\begin{itemize}
\item 
$ Q_0^n = R^1\pi_*(E(-C_0)) $ is a  torsion sheaf of length  
$$d_0= \frac{ \sum_in_i(n_i+1)}{2} + m_0;$$
\item 
$Q_n^0 = R^1\pi_*(E(-C_\infty))$ is a  torsion sheaf of length 
$$d_n = \frac{  \sum_j n_j( n_j-1 )}{2} + m_0;$$  $Q_0^n$ and $Q_n^0 $ are isomorphic away from $\eta= z_i$;
\item 
$Q_i^{n-i} = R^1\pi_*(E_i^{n-i})$ is a torsion sheaf of length 
$$ d_i=  \frac{  \sum_j n_j( n_j-1 )}{2} + m_0+\sum_{\ell=i+1}^n m_\ell  = \frac{ \sum_j n_j( n_j+1 )}{2} + m_0-\sum_{\ell=1}^i m_\ell,$$ 
\item 
$P_i^{n-i-1} = R^1\pi_*(E_i^{n-i-1})$ is a rank one sheaf of degree
$$d_i= \frac{  \sum_j n_j( n_j+1 )}{2} + m_0 - (\sum_{\ell=1}^i m_\ell). $$ 
 \end{itemize}
 We note that the $d_i$ must be positive or zero.
 \end{proposition}
 
 \section{Resolutions and instanton bundles}
 
 \subsection{Rebuilding $E$ from the sheaves $P,Q$. }
 
 The previous section gave us sheaves $P_i^{n-i-1} $, $Q_i^{n-i}$ supported on $\bP^1$ from the bundle $E$. This section is devoted to showing that the sheaves $P_i^{n-i-1} $, $Q_i^{n-i}$  and the natural maps between them actually encode the bundle $E$. In the process, we will   provide a description of our instanton bundles $E$ which ties them closely to the structure of the solutions to Nahm's equations associated to the instantons. A first step will be a resolution of the restriction of the lift of $E$ to the diagonal in  the fibre product
\label{Page:fiber}
 \begin{equation} \label{fibre}
 \xymatrix{  &X\times_{\bP^1} X\ar[dl]_{\pi_1}\ar[dr]^{\pi_2}&\\ X\ar[dr]&&X\ar[dl]\\&\bP^{1}&
    \\ }
\end{equation} 
over $\bP^1$. The base coordinate here is $\eta$; one has fibre coordinates $\tilde \xi,\tilde \psi$ for the first factor, $\xi,\psi$ for the second, with $\tilde \xi\tilde \psi=\xi \psi  =\prod_i(\eta-z_i)$. Let $X_\Delta$ denote the diagonal $\tilde \xi =\xi , \tilde \psi = \psi $.
Let $ \tilde E_{i }^{s } $ denote $\pi_1^*( E_{i }^{s })$, and let $\tilde C, \tilde D, \tilde F$ etc. denote our usual divisors pulled back from the first factor, and $C, D,F$ our divisors pulled back from the second factor.

  On $X\times_{\bP^1} X$ we have a diagram

\begin{equation}\label{resolve} \begin{matrix} \tilde E^{ n}_{0}(  -C_\infty)\\ \oplus \\\tilde E^{0}_{ n} (  -C_0) \end{matrix} \buildrel{ A}\over{ \rightarrow} \begin{matrix}\tilde E^{ n}_{0}(-C_0 )\\ \oplus\\ \tilde E^{n-1}_{0}\\ \oplus \\\tilde E^{n-2}_{1}\\ \oplus \\...\\...\\... \\ \oplus\\ \tilde E_{n-1}^{0}\\ \oplus \\\tilde E^{0}_{n}(-C_\infty )\end{matrix}  \buildrel{B}\over{ \rightarrow} \qquad\begin{matrix}\\ \tilde E^{ n}_{0}\\ \oplus\\ \tilde E^{n-1}_{1}\\ \oplus \\\tilde E^{n-2}_{2}\\ \oplus \\...\\ \oplus \\ \tilde E_{n-1}^{1}\\ \oplus \\\tilde E^{0}_{n}\\ \end{matrix}\quad   \buildrel{C}\over{ \rightarrow} \tilde E|_{X_\Delta} \rightarrow 0.
\end{equation}
 Here $A$ in essence just maps into the top and bottom sheaf, and is given by 
 $$A=\begin{pmatrix} \xi &\tilde \xi\\0&0\\ ...\\0& 0\\ -\tilde \psi&-\psi \end{pmatrix}.$$ The $(n+1)N$ by $(n+2)N$ matrix $B$, meanwhile, is given by 
$$B=\begin{pmatrix} \psi &-1&0&0&...&  0&\tilde \xi\\ 0&1&-1&0&...&0&0\\ 0&0&1&-1&...&0&0\\...&...&...&...&...&...&...\\0&0&0&0&...&-1&0 \\-\tilde \psi&0&0&0&...&1&-\xi \end{pmatrix},$$
where we understand each entry is a multiple of the  $N\times N$ identity matrix. 
Note, that the maximal minors of $B$ are, up to sign, products (recall that each entry is a matrix) of $\xi-\tilde \xi $, $\psi-\tilde \psi $, or $\xi\psi-\tilde \xi \tilde \psi  = 0$.  The map $C$ is simply the sum of the entries.

\begin{lemma} This is a (partial) resolution of $\tilde E|_{X_\Delta}$. \end{lemma}

\begin{proof} That the composition of two consecutive maps gives zero is straightforward; one has $B\circ A = 0$, and the sum of the columns of $B$ along $X_\Delta$ gives zero. Let us consider surjectivity for the image of $B$ onto the kernel of the projection to $X_\Delta$. First,  suppose  that we are at a point where $\xi-\tilde\xi\neq 0$, and away from $\tilde C_\infty, C_\infty$, so that neither $\xi,\tilde \xi$ is infinite. Suppose that an element $H= (\alpha_1,..,\alpha_{n+1})$ of the third column satisfies $\sum_I\alpha_i = 0 $ over $X_\Delta$. We want an element $G= (\beta_1, ...\beta_{n+2}) $ mapping to $H$. Setting $\beta_1 = 0$, we find 
\begin{align*} \beta_{n+2} =& (\tilde\xi-\xi)^{-1}\sum_i\alpha_i\\
 \beta_2=& -\alpha_1 + \tilde \xi (\tilde\xi-\xi)^{-1}\sum_i\alpha_i\\
 \beta_3 =& -\alpha_2 -\alpha_1 + \tilde \xi (\tilde\xi-\xi)^{-1}\sum_i\alpha_i\\ 
...& 
\end{align*}
and so on; one sees also that the $\beta_i$ lie in the appropriate subspaces of $E$ over $\tilde C_0$; also, that going to the limit $(\tilde\xi-\xi)=0$, the expression remains finite. Likewise, starting at a point where $(\psi-\tilde\psi)\neq 0$, and away from $\tilde C_0, C_0$, one can in a similar way solve for $H$, setting $\beta_{n+2}$ to start with, one obtains expressions for $H$ that extend to the locus $(\tilde\psi-\psi)=0$, and which have the right behaviour over $\tilde C_\infty$. Second, near the points of intersection   $ C_0\cap \tilde C_\infty, C_\infty\cap\tilde C_0$, remarking that one must multiply the first and last columns of $B$ by $\psi^{-1}, \xi^{-1}$ respectively for the change in trivialisation, one again obtains surjectivity, essentially by the same formulae.

Now let us consider an element $G= (\beta_1, ...\beta_{n+2}) $ with $B(G)=0$. This forces $\beta_2=\beta_3=...\beta_{n+1}$, as well as 
$$ \psi\beta_1 + \tilde \xi \beta_{n+2} = \tilde \psi\beta_1 +\xi \beta_{n+2}= -\beta_2.$$
This, plus the relation $\xi\psi =\tilde\xi\tilde \psi$, allows to write $\beta_1 = \xi a + \tilde \xi b, \ \beta_{n+2} = \tilde \psi a -\psi b$ for suitable $a,b$. One checks that this allows for the suitable vanishing of the components of $\beta_i$ over $C_0, C_\infty$. \end{proof}

Now define a sheaf $S$ by 
\begin{equation} \label{Sdef}\begin{matrix} \tilde E^{ n}_{0}( -C _\infty)\\ \oplus \\\tilde E^{0}_{ n} ( -C_0) \end{matrix} \buildrel{ \tilde A}\over{ \rightarrow} \begin{matrix}\tilde E^{ n}_{0}(-C_0 )\\ \oplus \\\tilde E^{0}_{n}(-C_\infty )\end{matrix}     \rightarrow S \rightarrow 0,
 \end{equation}
where $\tilde A=
\begin{pmatrix}
\xi&\tilde{\xi}\\
-\tilde{\psi}&-\psi
\end{pmatrix}
$ is the first and last row of $A$. Our resolution is then equivalent to 

\begin{equation}\label{mainresolution}0\rightarrow  \begin{matrix}S\\ \oplus\\ \tilde E^{n-1}_{0}\\ \oplus \\\tilde E^{n-2}_{1}\\ \oplus \\...\\...\\... \\ \oplus\\ \tilde E_{n-1}^{0} \end{matrix}  \buildrel{\tilde B}\over{ \rightarrow} \qquad\begin{matrix}\\ \tilde E^{ n}_{0}\\ \oplus\\ \tilde E^{n-1}_{1}\\ \oplus \\\tilde E^{n-2}_{2}\\ \oplus \\... \oplus \\ \tilde E_{n-1}^{1}\\ \oplus \\\tilde E^{0}_{n}\\ \end{matrix}\quad  \rightarrow \tilde E|_{X_\Delta} \rightarrow 0.
\end{equation}

 Now take the direct image under $\pi_2$ of the  sequences (\ref{Sdef}, \ref{mainresolution}).  If one goes near $\eta=\infty$, by Theorem~\ref{instanton-to-bundle}, one is dealing with bundles trivial on each fibre, and with flags that are transverse in the global trivialization; this ensures that the zeroeth direct images of the first and second terms vanish near $\eta= \infty$, and so vanish globally. For the first direct images of the sheaves $\tilde E^{n-i-1}_{i},\tilde E^{n-i}_{i}$, because of the diagram \eqref{fibre},   one obtains the lifts from $\bP^1$ of the sheaves $  P^{n-i-1}_{i}, Q^{n-i}_{i}$. We denote these pull backs of the sheaves $Q_i^{n-i}, P_i^{n-i-1}$ from $\bP^1$ to $X$ by the same notation, dropping the $\pi_2^*$. Pushing the sequence defining $S$ to the second factor of our fibre product yields the  defining sequence
 
\begin{equation} \label{r-definition} \begin{matrix} Q^{ n}_{0}( -C _\infty)\\ \oplus \\Q^{0}_{ n} ( -C_0) \end{matrix} \xrightarrow{R^1(\tilde A)} \begin{matrix}Q^{ n}_{0}(-C_0 )\\ \oplus \\Q^{0}_{n}(-C_\infty )\end{matrix}     \rightarrow R^1(\pi_2)_* (S) \buildrel{def}\over{=} R \rightarrow 0
\end{equation}
 with 
\begin{equation}\label{r-matrix}R^1(\tilde A)=  \begin{pmatrix} \xi&\widehat M_\xi\\-\widehat M_\psi &-\psi\end{pmatrix}\end{equation}
 where $\widehat M_\xi$ and $\widehat M_\psi$ describe the effect on cohomology of $\times \xi : E_{ n}^{0}\rightarrow E^{ n}_{0}$ and $\times \psi:E^{ n}_{0}\rightarrow E_{ n}^{0}$.  Note then that the support of $R$ is then included in the support of the  pullbacks of $Q^{ n}_{0},Q_{ n}^{0}$, so that $R$ is a torsion sheaf.
 
\begin{theorem} The  resolution  (\ref{mainresolution}) gives under pushdown to the second factor in the fibre product (\ref{fibre}) the following short exact sequence  \eqref{sequence_RQ} over $X$.

\begin{equation} \label{sequence_RQ}\xymatrix{&&&R\ar[dr]&\\ &&&\oplus &Q_{ 0}^{n}\\  &&& P^{ n-1}_{0 }\ar[ur]\ar[dr] &\oplus \\
&& &\oplus & Q^{n-1}_{ 1}
\\  & && P^{n-2}_{ 1} \ar[ur]\ar[dr]&\oplus\\ 
&&&\oplus&Q^{n-2}_{ 2}
\\0\ar[r]& E\ar[r]&&\vdots\ar[ur]\ar[dr]&\vdots&\ar[r]&0
\\ &&&\oplus&Q^1_{n-1}
\\&&& P_{n-1}^{0}\ar[ur]\ar[dr] &\oplus 
\\ && &\oplus &  Q^{0}_{n}\\ 
&&&R\ar[ur]&\\} \end{equation}

In this zig-zag diagram, each column defines a term in a short exact sequence; we have arrows only for non-zero maps between the terms of our vertical sums; if there is no arrow, that component of the map is zero. The $R$ at the top and the bottom of the diagram are identified; the repetition is to avoid crossing arrows. The map $R\rightarrow Q_{0}^{n} \oplus  Q^{ 0}_{n}$, referring to \ref{r-definition},  is induced by the map 
$Q_{0}^{ n} ( -C_0)\oplus Q^{0}_{n}(-C_\infty )\rightarrow  Q_{0}^{n} \oplus  Q^{ 0}_{n}$ given by 
\begin{equation}\label{zz}\begin{pmatrix} \psi&\widehat M_\xi\\-\widehat  M_\psi &-\xi\end{pmatrix}.\end{equation}
\end{theorem}
The result  follows from the fact that the relevant zeroeth direct images vanish for the resolution  (\ref{mainresolution}), using  the GRR calculations in 3.3 above. Note that multiplying (\ref{zz}) on the right by $R_1(\tilde A)$ gives a diagonal matrix with entries $\xi\psi -\widehat M_\xi \widehat M_\psi, \xi\psi -\widehat M_\psi\widehat  M_\xi$ which acts by zero; this holds as a consequence of the relation $\xi\psi =\tilde \xi\tilde \psi =\eta$ on the fibre product, recalling that the support of $ Q^{ 0}_{n},  Q_{ 0}^{n}$ is a divisor.

\subsection{The structure of the sheaves $P, Q$}

Again, recall that  our sheaves $Q_i^{n-i}, P_i^{n-i-1}$ are pulled back from $\bP^1$; again, we denote the sheaf on $\bP^1$, and its pull back to $X$ by the same symbol, to lighten notation.

We saw in the preceding section that the sheaves $P, Q$ (and the sheaf $R$ derived from them) determined a bundle $E$.
The aim now is to determine the structure of the sheaves $P, Q$ that will give back a suitable instanton bundle $E$. We summarise next some of the properties of these sheaves

\begin{theorem} \label{P,Q-properties}
\begin{itemize}
\item The sheaf $Q^{n-i}_{i}, i= 0,..,n$ is a torsion sheaf over $\bP^1$, supported away from $\eta= \infty$; it has length
$$ d_i=  \frac{  \sum_j n_j( n_j-1 )}{2} + m_0+\sum_{\ell=i+1}^n m_\ell  = \frac{ \sum_j n_j( n_j+1 )}{2} + m_0-\sum_{\ell=1}^i m_\ell,$$ 

\item The sheaf  $P^{n-i-1}_{i}, i= 0,.., n-1$ is a  rank one sheaf over $\bP^1$, of degree
$$d_i= \frac{  \sum_j n_j( n_j+1 )}{2} + m_0 - (\sum_{\ell=1}^i m_\ell). $$

\item There are exact sequences of  maps between the sheaves 
\begin{align}\label{p-to-q-sequence} 0\rightarrow \pO(m_i) \rightarrow &P_{i}^{n-i-1}\rightarrow Q_{i+1}^{n-i-1} \rightarrow 0,\\
 0\rightarrow \pO  \rightarrow &P_{i}^{n-i-1}\rightarrow Q_{i }^{n-i }\rightarrow 0.\end{align}
(Note that this implies that the torsion locus of the $P_{i }^{n-i-1}$ must lie in the intersection of the supports of $ Q_{i+1}^{n-i-1}$ and $Q_{i }^{n-i }$.)

\item Building the sheaf $R$ from (\ref{r-definition}) and (\ref{r-matrix}), and then setting 
$$\mathcal P = (\oplus_{i=0}^{n-1} P_{i}^{n-i-1})\oplus R,\quad \mathcal Q =  \oplus_{i=0}^{n} Q_{i}^{n-i },$$
we have that for all $x\in X$, the map of Tor-groups $Tor_1(\bC_x,\mathcal P) \rightarrow Tor_1(\bC_x,\mathcal Q)$ induced by $\tilde B$ is injective.

\end{itemize}

Any such family of $P_{i}^{n-i-1}, Q_{i }^{n-i }$, with maps between them as  above, and the map (\ref{r-matrix}) defining $R$,
satisfying the $Tor$-condition, defines a bundle $E$ on $X$ satisfying the constraints of Theorem  \ref{instanton-to-bundle}.
\end{theorem}

\begin{proof} The first two statements are covered in the proposition (\ref{degreeofQ}) in the preceding section. For the third, 
we have over $X$ the exact sequence of sheaves 
$$ 0\rightarrow E_{i-1}^{n-i}\rightarrow E_i^{n-i}\rightarrow \pO(m_i)|C_0 \rightarrow 0,$$
$$ 0\rightarrow E_i^{n-i-1}\rightarrow E_i^{n-i}\rightarrow \pO |C_\infty \rightarrow 0.$$
Taking a pushdown to $\bP^1$ gives the third statement, recalling that the $Q_i^{n-i}$ are torsion, supported away from $\eta= \infty$.

For the fourth item, from the exact sequence (\ref{sequence_RQ}): $0\rightarrow E\rightarrow \mathcal P \rightarrow \mathcal Q\rightarrow 0$, and the fact that $E$ is locally free at $x$ if and only if $Tor_1(\bC_x,E)= 0$, it follows from the $Tor$ long exact sequence that $E$ is locally free iff the map $Tor_1(\bC_x,\mathcal P) \rightarrow Tor_1(\bC_x,\mathcal Q)$ induced by $\tilde B$ is injective. (We will see below (Remark (\ref{tor2}) ) that $Tor_2(\bC_x,\mathcal Q) = 0$.)

Conversely, given $P_{i}^{n-i-1}, Q_{i }^{n-i }$, and the maps between them, given the matrix (\ref{r-matrix}) we can build the sheaf $R$, and so $ \mathcal P, \mathcal Q$ and then $E$.   The $Tor$-condition then ensures that $E$ is locally free. The flags   $E_{0, j}$  along $C_0 $ are given by sections of $E$ along $C_0$ lying in $\oplus_{i\leq j} P_{i}^{n-i-1}$, and similarly
$E^{\infty,k}$ by sections along $C_\infty$ lying in $\oplus_{i\leq k} P^{i}_{n-i-1}$. As the sheaves $P_{i}^{n-i-1}$ are lifted from $\bP^1$, the bundle $E$ is trivial on lines $\eta = c$ away from the support of $\mathcal Q$.
\end{proof}

\begin{lemma} From the sequences  (\ref{p-to-q-sequence}) above, one has 
\begin{align*}
h^0(\bP^1,P_{i }^{n-i-1})=& h^0(\bP^1,P_{i }^{n-i-1}(-1))+1 =  h^0(\bP^1,Q_i^{n-i})+1,\\h^1(\bP^1,P_{i }^{n-i-1}) =&h^1(\bP^1,P_{i }^{n-i-1}(-1)) = h^1(\bP^1,Q_i^{n-i})= 0.
\end{align*}
\end{lemma}
 By taking resolutions of the sheaves $P$, $Q$, and the maps between them, we see that the sheaves $P$, $Q$ are encoded by certain matrices. The next proposition summarises this.. 
\begin{proposition}\label{matrices} (Matrices encoding the sheaves $P$, $Q$, and the maps between them.) 
\begin{itemize}[leftmargin=*]
\item
The sheaves $P_{i }^{n-i-1}, Q_i^{n-i}$ have resolutions:
\begin{align}\label{resol-P-Q}
0\rightarrow \pO(-1)^{d_i} &
\xrightarrow{\begin{pmatrix} \eta\bI -\beta_i\\-\gamma_i\end{pmatrix}} 
\pO^{d_i+1}\rightarrow P_{i}^{n-i-1}\rightarrow 0,\\
0\rightarrow \pO(-1)^{d_i} &
\xrightarrow{\begin{pmatrix}\eta\bI -\beta_i \end{pmatrix} }
\pO^{d_i}\rightarrow Q_{i}^{n-i}\rightarrow 0.
\end{align}
\item This induces a resolution for the sheaf $R$:
\begin{equation}\label{resol-R}
\begin{matrix}
\pO(-F-C_0)^{d_0}\\\oplus \\\pO(-F-C_\infty)^{d_n}\\ \oplus \\ \pO (-C_\infty)^{d_0}\\\oplus\\ \pO (-C_0)^{d_n}
\end{matrix} 
\xrightarrow{\left(\begin{smallmatrix} \eta\bI -\beta_0 &0& \xi& M_\xi\\0&\eta\bI -\beta_n& - M_\psi &-\psi\end{smallmatrix}\right)}
\begin{matrix}\pO(-C_0)^{d_0} \\ \oplus \\ \pO(-C_\infty)^{d_n} \end{matrix}
\longrightarrow R\longrightarrow 0.\end{equation}

\item The maps from the  $ P$ to the  $Q$ in diagram  (\ref{sequence_RQ}) induce commuting diagrams on  the resolutions:

\begin{equation} \label{PtoQ}\xymatrix{0\ar[r]&\pO(-1)^{d_{i+1}} \ar[r]&\pO^{d_{i+1}}\ar[r] &Q_{i+1}^{n-i-1}\ar[r]& 0\\ 
0\ar[r] & \pO(-1)^{d_i} \ar[r]\ar[u]_{\begin{pmatrix}A_i\end{pmatrix}}\ar[d]^\bI&\pO^{d_i+1}\ar[r] \ar[u]_{\begin{pmatrix}A_i,&\alpha_i\end{pmatrix}}\ar[d]^{\begin{pmatrix}\bI,&0\end{pmatrix}}&P_{i}^{n-i-1}\ar[r]\ar[u]\ar[d]& 0\\
0\ar[r] & \pO(-1)^{d_i} \ar[r]&\pO^{d_i}\ar[r] &Q_{i}^{n-i }\ar[r]& 0\\}
 \end{equation}
 \item The maps from $R$ to $Q_0^n$, $Q_n^0$  in diagram  (\ref{sequence_RQ}) induce a diagram on resolutions
 \begin{equation}\label{RtoQ} 
 \xymatrix{ &\pO^{d_{0}}\ar[r] &Q_{0}^{n}\ar[r]& 0\\ \\
   { \begin{matrix}
\pO(-F_\infty-C_0)^{d_0} \oplus  \pO(-F_\infty-C_\infty)^{d_n} \\ \oplus\\ \pO (-C_\infty)^{d_0} \oplus \pO (-C_0)^{d_n}\end{matrix}
 } \ar[r]  &{\begin{matrix}\pO(-C_0)^{d_0} \\ \oplus \\ \pO(-C_\infty)^{d_n} \end{matrix}}\ar[r] \ar[uu]_{(\psi,M_\xi)}\ar[dd]^{(-M_\psi,-\xi)}&R \ar[r]\ar[uu]\ar[dd]& 0.\\ \\
  &\pO^{d_n}\ar[r] &Q_{n}^{0 }\ar[r]& 0}
\end{equation}
 
 \item The commutativity of the diagrams (\ref{PtoQ}) above gives the relation on matrices:
 \begin{align}\label{mat} 
 &\beta_{i+1} A_i - A_i \beta_i - \alpha_i\gamma_i = 0,& i&=0,..,n-1. 
 \end{align}  
 
  \item The  diagram  (\ref{RtoQ}) above gives the relation on matrices:
 \begin{equation}\label{mat2} M_\psi\beta_{0}   = \beta_n M_\xi. \end{equation}

 \item Suppose that one has the relations (\ref{mat}) on the matrices. The exactness of the  sequence (\ref{resol-P-Q}) and the fact that the first term is locally free is equivalent to asking that the sequence 
\begin{equation}\label{exact-p-q}0\rightarrow \pO(-1)^{d_{i }} \xrightarrow{\begin{pmatrix}\eta-\beta_i\\ -\gamma_i\\-A_i \end{pmatrix}}  \pO^{d_{i}+1}\oplus \pO^{d_{i+1}}(-1)\xrightarrow{\begin{pmatrix}A_i, \alpha_i, \eta-\beta_{i+1}   \end{pmatrix}}  \pO^{d_{i+1}}\rightarrow 0 
\end{equation}
 be exact (pointwise) for all $\eta$.
 
\end{itemize}
\end{proposition}
\begin{proof}
We have on $\bP^1\times\bP^1$ the resolution of the diagonal:
$$ 0\rightarrow \pO(-1,-1)  \rightarrow \pO\rightarrow \pO|_\Delta\rightarrow 0,$$
and so, tensoring by the pull back of $P_{i}^{n-i-1}$, $Q_i^{n-i}$ from one factor, and pushing down to the other, one has resolutions 
\begin{align*}
0\rightarrow \pO(-1)^{d_i} &{\xrightarrow{\begin{pmatrix}G_{i,1}\eta +G_{i,0} \end{pmatrix} }} \pO^{d_i+1}\rightarrow P_{i}^{n-i-1}\rightarrow 0,\\
0\rightarrow \pO(-1)^{d_i} &{\xrightarrow{\begin{pmatrix}H_{i,1}\eta +H_{i,0} \end{pmatrix} }} \pO^{d_i}\rightarrow Q_{i}^{n-i}\rightarrow 0.
\end{align*}
The dimensions $d_i$ were computed in Proposition (\ref{degreeofQ}) above:
$$d_i = \frac{ \sum_j n_j( n_j+1 )}{2} + m_0-\sum_{\ell=1}^i m_\ell.$$
One can normalize to $G_{i,1} =H_{i,1} = \bI$; and since the map $P_i^{n-i-1}(-1)\rightarrow Q_i^{n-i}(-1)$ induces an isomorphism on cohomology, we can further normalize to the form given in the theorem.

For $Q_{0}^{n}, Q_{n}^{0}$ we have  
\begin{align*}
0\rightarrow \pO(-1)^{d_0} &
\xrightarrow{\begin{pmatrix}\eta\bI -\beta_0 \end{pmatrix} }
\pO^{d_0}\rightarrow Q_{0}^{n}\rightarrow 0,\\
0\rightarrow \pO(-1)^{d_n} &
\xrightarrow{\begin{pmatrix}\eta\bI -\beta_n\end{pmatrix} }
\pO^{d_n}\rightarrow Q^{0}_{n}\rightarrow 0,   
\end{align*}
which in turn, we insert into the expression (\ref{r-definition}) for $R$ to obtain the description  of $R$  in the proposition.

The maps on the sheaves naturally extend to their resolutions; hence the diagrams in the proposition, and the ensuing   commutation relations.

Finally, if one sets
\begin{equation} \xymatrix{ &&&&C\\                            
0\ar[r]&\pO(-1)^{d_{i+1}} \ar[rr]^{\begin{pmatrix} \eta-\beta_{i+1}    \end{pmatrix}}&&\pO^{d_{i+1}}\ar[r] &Q_{i+1}^{n-i-1}\ar[r]\ar[u]& 0\\ \\  
0\ar[r] & \pO(-1)^{d_i} \ar[rr]^{\begin{pmatrix} \eta-\beta_{i} \\-\gamma_i  \end{pmatrix}}\ar[uu]_{\begin{pmatrix}A_i\end{pmatrix}} &&\pO^{d_i+1}\ar[r] \ar[uu]_{\begin{pmatrix}A_i,\alpha_i\end{pmatrix}} &P_{i}^{n-i-1}\ar[r]\ar[uu]& 0,\\
&&&&K\ar[u] }
 \end{equation}
with $C, K$ the cokernel and kernel respectively, one sees that $C$ vanishes if the sequence (\ref{exact-p-q}) surjects onto the last term; likewise $K$ is locally free if the first term injects.
\end{proof}

\begin{remark}\label{tor2} The fact that the sheaves $Q_i^{n-i}$ have a resolution of length two tells us that $Tor_2(\bC_x, Q_i^{n-i})$ vanishes, as announced.\end{remark}

\subsection{A monad description of $E$}

The matrices $A_i, \beta_i, M_\xi, M_\psi$ occurring in Proposition (\ref{matrices}) can also be assembled into the  matrices defining a monad description of $E$: if we write schematically
$$\pP = (\oplus_{i=0}^{n-1} P_{i}^{n-i-1})\oplus R,\qquad \pQ= \oplus_{i=0}^{n } Q_{i}^{n-i };$$ then
Proposition (\ref{matrices}) gives compatible resolutions of $\pP, \pQ$, which we write schematically as 
\begin{equation}\begin{matrix}\pF&\buildrel{\mu}\over{\rightarrow}&\pA&\buildrel{\alpha}\over{\rightarrow} &\pB&\rightarrow &\pP&\rightarrow &0\\
&&\downarrow\beta&&\downarrow\delta&&\downarrow\epsilon\\
0&\rightarrow &\pC&\buildrel{\gamma}\over{\rightarrow} &\pD&\rightarrow &\pQ&\rightarrow &0,\end{matrix}\end{equation}
with $\pA, \pB, \pC, \pD$ sums of line bundles on $X$, and the maps between them given  by Proposition (\ref{matrices}).
The extra term $\pF$, a sum of two line bundles,  is obtained from extending the resolution of $R$; it maps to the four sheaves in the resolution of $R$ that form part of $\pA$ by a matrix
$$\mu = \begin{pmatrix} 1&0\\0&1\\\psi&M_\xi\\-M_\psi& -\xi\end{pmatrix}.$$
We also have an exact sequence
$$0\rightarrow E\rightarrow\pP\xrightarrow{\epsilon} \pQ\rightarrow 0.$$
Some diagram chasing \cite[Sec.~4.2]{Cherkis:2017pop} shows that $E$ is then the cohomology $\ker /\mathrm{Im} $ of the monad
\begin{equation}\pA\xrightarrow{\begin{pmatrix}\alpha&-\beta\end{pmatrix}} \pB\oplus \pC\xrightarrow{\begin{pmatrix}\delta\\\gamma\end{pmatrix}}\pD \rightarrow 0.\end{equation}
(The map $\alpha$ is not necessarily  injective, essentially because of the resolution of $R$, but this is of no consequence.) The matrices involved in the monad, or equivalently in the resolutions and the maps between them, encode the  Nahm complex associated to the bundle $E$. We will return to this later.

Given the matrices, one can also reconstruct our sheaves $P, Q, R$ and so  $E$; this is straightforward; the one non-immediate thing that one must check is that $E$ is locally free, or alternately, for any point $x$ with coordinates $\xi, \psi$, the vector space $Tor_1(\bC_x, E)$ is trivial. Referring  to the standard long exact sequence for $Tor$, and to the fact that we have just built resolutions of our sheaves, the $Tor$-condition amounts to asking that the map 
 $\beta: \ker(\alpha)/{\mathrm{Im}}(\mu) \rightarrow \ker(\gamma)$  be injective at every point.

\subsection{The structure of the sheaves $S$, $R$, and their decompositions.}

We have a sheaf $R$ playing a particular role in our construction, essentially defining a correspondence between  the two sheaves $Q^0_n$ and $Q_0^n$. We will see in the next subsection from the point of view of the Nahm transform, this correspondence will give  the quasiperiodicity of Nahm flows on the circle associated to the instanton. It is more natural, however, in the Nahm transform for $\tnk$, to have  the correspondence  decompose  into a series of correspondences, one associated with each `NUT', that is each singular point of the fibre over $\eta = z_i$ in the fibration $\tnk\rightarrow \bR^3$. This is the picture considered in \cite{Cherkis:2008ip}, for example. The sheaf $R$ in some sense handles all these singular points together. We now show how this correspondence can be factored from the algebro-geometric point of view that we have here.

  The sheaf $R$ is a torsion sheaf, with   the same support as $Q_0^n, Q^0_n$; as noted, it defines a correspondence between the two, which is an isomorphism away from the `NUTs', that is the points $\eta= z_i$. This correspondence can be decomposed into a product of correspondences, ``supported'' (in the sense of defining a natural isomorphism away from the support) at each $z_i$  in turn, as follows. 

Recall that  in the sequence \eqref{resolve} above over our fibre product, there is a portion of the map $B$:
\begin{equation} \label{bee} \begin{matrix} \tilde E^{ n}_{0}( -C _0)\\ \oplus \\\tilde E^{0}_{ n} ( -C_\infty) \end{matrix}\qquad 
\xrightarrow{\begin{pmatrix} \psi& \tilde \xi\\ -\tilde \psi&-\xi\end{pmatrix}}\qquad \begin{matrix}\tilde E^{ n}_{0} \\ \oplus \\\tilde E^{0}_{n} \end{matrix}.      
\end{equation}
This portion of the map  $B$ factors through the sheaf $S$  in the sequence (\ref{resolve}) above. There is an equivalence of divisors (away from $\eta=\infty$) $C_\infty = C_0 +\sum_iD_{i,\xi}$ given by the function $\xi$, and there is a similar one $ \tilde C_\infty = \tilde C_0 +\sum_i\tilde D_{i,\xi}$ given by $\tilde \xi$;  in the same way, there is an isomorphism $C_0 = C_\infty +\sum_iD_{i,\psi} $, given by $\psi$  (away from $\eta=\infty$) , and  $\tilde C_0 = \tilde C_\infty +\sum_i\tilde D_{i,\psi} $ given by $\tilde \psi$. We then have
  isomorphisms away from $\eta=\infty$:
\begin{align*} 
\tilde E^n_0(-C_0) & \xrightarrow{\tilde\xi^{-1}} \tilde E_n^0(-C_0+\sum_i\tilde D_{i,\xi}),\\
\tilde E_n^0(-C_\infty)& \xrightarrow{\xi} \tilde E_n^0(-C_0-\sum_iD_{i,\xi}),\\
\tilde E^n_0& \xrightarrow{\tilde\xi^{-1}\psi^{-1}} \tilde E_n^0(-C_0+C_\infty+\sum_i  (D_{i,\psi}+ \tilde D_{i,\xi})),\end{align*}
and with these, the diagram (\ref{bee}) becomes
\begin{equation}\label{beeprime}  \begin{matrix} \tilde E_n^0(-C_0+\sum_i\tilde D_{i,\xi})\\ \oplus \\\tilde E_n^0(-C_0-\sum_iD_{i,\xi}) \end{matrix} \qquad\buildrel{ \widetilde B}\over{ \longrightarrow} \qquad \begin{matrix}  \tilde E_n^0(-C_0+C_\infty+\sum_i  (D_{i,\psi}+ \tilde D_{i,\xi}))\\ \oplus \\\tilde E^{0}_{n} \end{matrix},    
\end{equation}
 with 
 $$\widetilde B = \begin{pmatrix} 1&\prod_i(\eta- z_i)^{-1}\\  -\prod_i(\eta- z_i)&-1\end{pmatrix}.$$
 The diagram (\ref{bee}) or (\ref{beeprime}) should be thought of as   defining a correspondence between its top right hand element and its bottom right hand element, by $(a, 0)\mapsto (0,  -\prod_i(\eta- z_i)a)$, that is pulling back and pushing forward ; since the map ${ \widetilde B}$ is degenerate, pulling back to the top left hand or the bottom left hand to define the correspondence makes no difference; in either case, the correspondence is given  by multiplication by $\prod_i(\eta- z_i)$ or its inverse, depending on the direction.
  
 This can then be split into a sequence of  such correspondence diagrams, two by two blocks, with the bottom row of one block being the top row of the next. For the $i$-th block the correspondence should be multiplication by $\eta- z_i$ or its inverse. Thus we can factor the correspondence  as follows.
 
 \begin{align}\label{corresp} 
\begin{smallmatrix} \tilde E_n^0(-C_0+\sum_i\tilde D_{i,\xi})\\ \oplus \\\tilde E_n^0(-C_0-D_{k,\xi}+\sum_{i\leq k-1}\tilde D_{i,\xi} ) 
\end{smallmatrix}\quad &\xrightarrow{B_k}\quad \begin{smallmatrix}  \tilde E_n^0(-C_0+C_\infty+\sum_i  (D_{i,\psi}+ \tilde D_{i,\xi}))\\ \oplus \\\tilde E_n^0(-C_0+C_\infty +\sum_{i}D_{i,\psi}+ \sum_{i\leq k-1}\tilde D_{i,\xi}), 
\end{smallmatrix}    \nonumber  
\\\nonumber
 &\\\nonumber
  \begin{smallmatrix} \tilde E_n^0(-C_0-D_{k,\xi}+\sum_{i\leq k-1}\tilde D_{i,\xi} )\\ \oplus \\\tilde E_n^0(-C_0-\sum_{i>k-2}D_{i,\xi}   +\sum_{i\leq k-2}\tilde D_{i,\xi} ) \end{smallmatrix}\quad &\xrightarrow{B_{k-1}}\quad \begin{smallmatrix}  \tilde E_n^0(-C_0+C_\infty +\sum_{i}D_{i,\psi}+ \sum_{i\leq k-1}\tilde D_{i,\xi}) \\ \oplus \\ \tilde E_n^0(-C_0+C_\infty +\sum_{i}D_{i,\psi}+ \sum_{i\leq k-2}\tilde D_{i,\xi})), \end{smallmatrix}     
\\\nonumber
\vdots\qquad &\vdots\qquad \vdots\\\nonumber
  \begin{smallmatrix} \tilde E_n^0(-C_0 -\sum_{i>2}D_{i,\xi}+ \sum_{i\leq 2}\tilde D_{i,\xi} )\\ \oplus \\\tilde E_n^0(-C_0 -\sum_{i>1}D_{i,\xi}+ \tilde D_{1,\xi} ) \end{smallmatrix}\quad &\xrightarrow{B_{2}}\quad \begin{smallmatrix}  \tilde E_n^0(-C_0+C_\infty+\sum_{i}D_{i,\psi} + \sum_{i\leq 2}\tilde D_{i,\xi}))\\ \oplus \\\tilde E_n^0(-C_0+C_\infty+\sum_{i}D_{i,\psi}+  \tilde D_{1,\xi})), \end{smallmatrix}      
\\\nonumber
&\\
 \begin{smallmatrix} \tilde E_n^0(-C_0-\sum_{i>1}D_{i,\xi}+ \tilde D_{1,\xi}  )\\ \oplus \\ \tilde E_n^0(-C_0 -\sum_{i>0}D_{i,\xi})\simeq  \tilde E_n^0(-C_\infty) \end{smallmatrix}\quad &\buildrel{B_{1}}\over{ \longrightarrow}\quad \begin{smallmatrix} \tilde E_n^0(-C_0+C_\infty+ \sum_{i}D_{i,\psi}+ \tilde D_{1,\xi})\\ \oplus \\\tilde E_n^0(-C_0+C_\infty+\sum_{i}D_{i,\psi})\simeq \tilde E_n^0. \end{smallmatrix}      
\end{align}
Here $B_j$ is the matrix
$$B_j = \begin{pmatrix} 1& (\eta-  z_j)^{-1}\\ -(\eta-  z_j)&-1\end{pmatrix}.$$

Let us define functions:
\begin{align} \xi_1 &= \xi,& 
\psi_1 &= \xi_1^{-1}(\eta-z_1),& 
\xi_2 &= \psi_1^{-1},& 
&\ldots\\
\psi_j &= \xi_j^{-1}(\eta-z_j),&
 \xi_{j+1} &= \psi_j^{-1},&
 &\ldots&
  \xi_k\psi_k&= (\eta-z_j).\nonumber
  \end{align}
\begin{lemma} We have that $\psi_k = \psi$,  the divisor of $\xi_j$ is $C_0-C_\infty + \sum_{i\geq j} D_{i,\xi} - \sum_{i< j} D_{i,\psi}$, and that of $\psi_j$ is $-C_0+C_\infty  - \sum_{i> j} D_{i,\xi} + \sum_{i\leq j} D_{i,\psi}$.\end{lemma}

We consider the $j$-th map in (\ref{corresp}) above. We change trivialisations, by
  isomorphisms away from $\eta=\infty$:
\begin{align*} 
\scriptstyle \tilde{E}^n_0(-C_0-\sum_{i>j }( D_{i,\xi}+  \tilde D_{i,\xi}) +\sum_{i<j} (\tilde D_{i,\psi} + \tilde D_{i,\xi})    ) & \xrightarrow{\tilde \xi_j^{-1}}\ 
\scriptstyle \tilde E_n^0(-C_0 -\sum_{i>j } D_{i,\xi}+ \sum_{i\leq j }\tilde D_{i,\xi} ),\\
\scriptstyle \tilde E_n^0(- C_\infty+ \sum_{i<j}\tilde D_{i,\xi})& \xrightarrow{\xi_j} \ 
\scriptstyle \tilde E_n^0(-C_0 -\sum_{i\geq j } D_{i,\xi}+ \sum_{i<j }\tilde D_{i,\xi} ),\\
\scriptstyle \tilde E_n^0(-C_0+C_\infty+\sum_{i}D_{i,\psi} + \sum_{i\leq j}\tilde D_{i,\xi})& \xrightarrow{\psi_j\tilde \xi_j } \ 
\scriptstyle \tilde E^n_0( \sum_{i>j}(D_{i,\psi}-\tilde D_{i,\xi})  + \sum_{i<j}( \tilde D_{i,\psi} +\tilde D_{i,\xi})),
 \end{align*} 
 
this modifies  $B_j$ to $\hat B_j$:
\begin{equation}\label{ch}\hat B_j= \begin{pmatrix} \psi_j\tilde \xi_j & 0\\ 0& 1\end{pmatrix}\begin{pmatrix} 1& (\eta-  z_j)^{-1}\\ -(\eta-  z_j)&-1\end{pmatrix}\begin{pmatrix}  \tilde \xi_j ^{-1}& 0\\0&\xi_j \end{pmatrix} = \begin{pmatrix} \psi_j&\tilde \xi_j\\  -\tilde \psi_j&-\xi_j\end{pmatrix}. \end{equation} 

Now take a direct image, as in (\ref{mainresolution}).  Recalling $\tilde E_n^0(-C_0+C_\infty+\sum_{i}D_{i,\psi})\simeq \tilde E_n^0$, we set 
\begin{equation} Q_{n,j}^0 = R^1(\pi_2)_*(\tilde E_n^0(\sum_{i\leq j}\tilde D_{i,\xi})).   \end{equation}
Let
\begin{align} \label{Mxij} M_{\xi, j} :& Q_{n,j-1}^0 \rightarrow  Q_{n,j}^0 \\ M_{\psi, j} :& Q_{n,j}^0 \rightarrow  Q_{n,j-1}^0  \nonumber  \end{align}
be the maps induced on the direct image by multiplication by $\xi_j, \psi_j$ respectively.

Then the correspondences (\ref{corresp}) above project to 
\begin{equation} \label{Qcorresp}\begin{matrix}  
  Q_{n,j}^0(-C_0 -\sum_{i>j}D_{i,\xi}) \\  \oplus\\
 Q_{n,j-1}^0  (-C_0 -\sum_{i>j-1}D_{i,\xi})
\end{matrix}
  \xrightarrow{R^1(B_j)} \begin{matrix}   Q_{n,j}^0 \\
\oplus\\ Q_{n,j-1}^0  \end{matrix}.
 \end{equation} 
 
 With suitable changes of trivialisation, as above (\ref{ch}), the matrix ${R^1(B_j)}$ becomes:
 $$\begin{pmatrix} \psi_j& M_{\xi,j}\\-M_{\psi,j} &-\xi_j\end{pmatrix}.$$

 We can in a similar fashion, factor the sequence that defines $S$:
 $$  \begin{matrix} \tilde E^{ n}_{0}( -C _\infty)\\ \oplus \\\tilde E^{0}_{ n} ( -C_0) \end{matrix} \buildrel{ \widetilde A}\over{ \rightarrow} \begin{matrix}\tilde E^{ n}_{0}(-C_0 )\\ \oplus \\ \tilde E^{0}_{n}(-C_\infty )\end{matrix}     \rightarrow S \rightarrow 0.
 $$
 Combining this with the isomorphisms away from $\eta = \infty$
 \begin{align*} 
 \tilde E_n^0(-C_\infty)& \buildrel{\xi}\over{\rightarrow} \tilde E_n^0(-C_0-\sum_iD_{i,\xi}),\\
\tilde E^n_0(-C_0) & \buildrel{\tilde\xi^{-1}}\over{\rightarrow} \tilde E_n^0(-C_0+\sum_i\tilde D_{i,\xi}),\\
\tilde E^n_0(-C_\infty)& \buildrel{\tilde\psi \xi}\over{\rightarrow} \tilde E_n^0(-C_0-\sum_i  (D_{i,\xi}+ \tilde D_{i,\psi})),
\end{align*}
we have
$$  \begin{matrix}  \tilde E_n^0(-C_0-\sum_i  (D_{i,\xi}+ \tilde D_{i,\psi}))\\ \oplus \\\tilde E^{0}_{ n} ( -C_0) \end{matrix} \buildrel{ \hat A}\over{ \rightarrow} \begin{matrix}\tilde E_n^0(-C_0+\sum_i\tilde D_{i,\xi})\\ \oplus \\\tilde E_n^0(-C_0-\sum_iD_{i,\xi})\end{matrix}     \rightarrow S \rightarrow 0,
 $$
 with 
 $$\widehat A = \begin{pmatrix} \prod_i(\eta- z_i)^{-1}& 1\\ -1& -\prod_i(\eta- z_i) \end{pmatrix}.$$
 This factors into a sequence of sheaves:
 \begin{align*}   
 \begin{smallmatrix}  \tilde E_n^0(-C_0-\sum_i  (D_{i,\xi}+ \tilde D_{i,\psi}))\\ \oplus \\\tilde E_n^0(-C_0-\sum_{i\leq k-1} (D_{i,\xi}+ \tilde D_{i,\psi})) \end{smallmatrix} &\xrightarrow{ \widehat A_k}  
 \begin{smallmatrix} \tilde E_n^0(-C_0+\sum_i\tilde D_{i,\xi})\\ \oplus \\\tilde E_n^0(-C_0-D_{k,\xi}+\sum_{i\leq k-1}\tilde D_{i,\xi} ) \end{smallmatrix}     \rightarrow {\scriptstyle S_k} \rightarrow \scriptstyle 0,\\
 &\\
 \begin{smallmatrix} \tilde E_n^0(-C_0-\sum_{i\leq k-1} (D_{i,\xi}+ \tilde D_{i,\psi}))\\ \oplus \\\tilde E_n^0(-C_0-\sum_{i\leq k-2} (D_{i,\xi}+ \tilde D_{i,\psi})) \end{smallmatrix} &\xrightarrow{ \widehat A_{k-1}}  \begin{smallmatrix} \tilde E_n^0(-C_0-D_{k,\xi}+\sum_{i\leq k-1}\tilde D_{i,\xi} )\\ \oplus \\\tilde E_n^0(-C_0-\sum_{i>k-2}D_{i,\xi}   +\sum_{i\leq k-2}\tilde D_{i,\xi})\end{smallmatrix}\rightarrow {\scriptstyle S_{k-1} }\rightarrow\ \scriptstyle 0, \\
 &....\\
 &....\\
  \begin{smallmatrix} \tilde E_n^0(-C_0-\sum_{i\leq 2} (D_{i,\xi}+ \tilde D_{i,\psi}))\\ \oplus \\\tilde E_n^0(-C_0-  (D_{1,\xi}+ \tilde D_{1,\psi})) \end{smallmatrix} &\xrightarrow{ \widehat A_{2}}  \begin{smallmatrix} \tilde E_n^0(-C_0 -\sum_{i>2}D_{i,\xi}+ \sum_{i\leq 2}\tilde D_{i,\xi} )\\ \oplus \\\tilde E_n^0(-C_0 -\sum_{i>1}D_{i,\xi}+ \tilde D_{1,\xi} ) \end{smallmatrix}   \rightarrow {\scriptstyle S_{2}} \rightarrow\ \scriptstyle 0, \\  
  &\\
    \begin{smallmatrix}\tilde E_n^0(-C_0-  (D_{1,\xi}+ \tilde D_{1,\psi})) \\ \oplus \\\tilde E_n^0(-C_0) \end{smallmatrix} &\xrightarrow{ \widehat A_{1}}  \begin{smallmatrix}\tilde E_n^0(-C_0-\sum_{i>1}D_{i,\xi}+ \tilde D_{1,\xi}  )\\ \oplus \\\tilde E_n^0(-C_0 -\sum_{i>0}D_{i,\xi}) \end{smallmatrix}   \rightarrow {\scriptstyle S_{1} }\rightarrow \scriptstyle\ 0. \\
 \end{align*} 
 
Here    
$$\widehat A_{j} = \begin{pmatrix}  (\eta- z_j)^{-1}&-1\\  1&  -(\eta- z_j) \end{pmatrix}.$$
Again one can change trivialisations: 
$$\begin{pmatrix}\tilde  \xi_j  & 0\\ 0&\xi_j^{-1}\end{pmatrix} \begin{pmatrix}  (\eta- z_j)^{-1}&-1\\ -1&  (\eta- z_j) \end{pmatrix}\begin{pmatrix}\tilde \psi_j \xi_j  & 0\\ 0& 1\end{pmatrix} =  \begin{pmatrix} \xi_j&\tilde\xi_j\\ -\tilde \psi_j&  -\psi_j\end{pmatrix}.$$
Taking the first direct image 
\begin{equation}\label{defrj}R_j {\buildrel{def}\over{=}}R^1(\pi_2)_*(S_j), \end{equation}  
map (\ref{Qcorresp}) factors through $R_j$:
$$ R_j \rightarrow \begin{matrix}  Q_{n,j}^0\\ \oplus \\ Q_{n,j-1}^0 \end{matrix},$$
and the correspondence 
  \begin{equation} \label{sequence R1}\xymatrix{&R\ar[dr]\ar[dl]&\\ Q_{0}^{n}&&Q^{0}_{n}.}\end{equation}
  between $Q^n_0$ and $Q_n^0$ factors as :
 \begin{equation} \label{sequence R2}
 \xymatrix@C=1.5em{& R_k\ar[dr]\ar[dl]&& R_{k-1}\ar[dr]\ar[dl]&&...\ar[dr]\ar[dl]&& R_1\ar[dr]\ar[dl]&\\
 Q_{0}^{n}&&Q^{0}_{n, k-1}&& Q^{0}_{n, k-2}&...& Q^{0}_{n, 1}&&  Q^{0}_{n}.   }\end{equation}
   
 Grothendieck-Riemann-Roch calculations, and resolving in the same way as for $Q_i^{n-i}$, etc., gives:
  \begin{proposition} \label{matricesQnj} The sheaves $ Q_{n,j}^0$ are lifts of torsion sheaves on $\bP^1$, of length
  $$ d_{n,j}= m_0 +\frac{1}{2} \sum_{i\leq j} (n_i^2 +n_i) + \sum_{i> j} (n_i^2 -n_i).$$
 Note that $d_{n,0} = d_n, d_{n,k} = d_0$. 
 The sheaves  $R_j ,Q_{n,j}^0$ have resolutions fitting into diagrams away from $\eta = \infty$:
 
 \begin{equation} 
 \xymatrix{ \pO(-1)^{d_{n,j}} \ar[r]^{\begin{matrix}\eta-\beta_{n,j}\end{matrix}}&\pO^{d_{n,j}}\ar[r] &Q_{n,j}^{0}\ar[r]& 0\\ \\
   { \begin{smallmatrix}
\pO( -C_0)^{d_{n,j}} \oplus  \pO( -C_\infty)^{d_{n,j-1}} \\ \oplus\\ \pO (-C_\infty -\sum_{i<j}D_{i,\psi})^{d_{n,j}}\\ \oplus \\ \pO (-C_0-\sum_{i>j}D_{i,\xi})^{d_{n,j-1}}\end{smallmatrix}
 } \ar[r]^G\ar[uu]_{(\begin{smallmatrix}\psi_j&M_{\xi_j}&\bI&0\end{smallmatrix})}\ar[dd]^{(\begin{smallmatrix}-M_{\psi_j}&-\xi_j&0&\bI\end{smallmatrix})}&{\begin{smallmatrix}\pO(-C_0-\sum_{i>j}D_{i,\xi})^{d_{n,j}}\\ \oplus \\ \pO(-C_\infty-\sum_{i<j}D_{i,\psi})^{d_{n,j-1}} \end{smallmatrix}}\ar[r] \ar[uu]_{(\begin{smallmatrix}\psi_j&M_{\xi_j}\end{smallmatrix})}\ar[dd]^{(\begin{smallmatrix}-M_{\psi_j}&-\xi_j\end{smallmatrix})}&R_j \ar[r]\ar[uu]\ar[dd]& 0\\ \\
 \pO(-1)^{d_{n,j-1}} \ar[r]_{\begin{matrix}\eta-\beta_{n,j-1}\end{matrix}}&\pO^{d_{n,j-1}}\ar[r] &Q_{n,j-1}^{0 }\ar[r]& 0\\ }
 \end{equation}
 where 
 $$G = \begin{pmatrix} \eta-\beta_{n,j}& 0 &\xi_j &M_{\xi_j}\\ 0& \eta-\beta_{n,j-1}&- M_{\psi_j}&-\psi_j\end{pmatrix}.$$
 The commutation of these diagrams yields the relations
 \begin{equation} \beta_{n,j} = M_{\xi_j}M_{\psi_j} + z_j,\quad \beta_{n,j-1} = M_{\psi_j}M_{\xi_j} + z_j.\end{equation}
 \end{proposition}

 \subsection{Bows}\label{Sec:Bows}
 For different boundary conditions,  moduli of solutions to Nahm's equations correspond to many different moduli spaces of solutions to the anti-self-duality equation and its various reductions; see e.g. \cite{Jardim} for a review.   {   The set of solutions we will want to study here arise via the Nahm transform from  $U(n)$ instantons on $\tnk$, and are called bow solutions. The boundary conditions that one imposes on the solutions are defined by   linear maps,  called, for representation theoretic reasons,  the fundamental and bifundamental data. }
 
 {\subsubsection{The Bow} The bow solutions which concern us are defined over a circle of perimeter $\ell$, with a linear parameter $s$  which we will take to be multi-valued, identifying the point with coordinate $s$ with the one with coordinate $s+\ell$. In this parametrization, we have the following data
\begin{itemize} 
\item There are marked points at $s=\lambda_i, i = 1,..,n$ on the circle (``$\lambda$-points''), where $0<\lambda_1<\lambda_2<\ldots<\lambda_n<\ell$; the $\lambda_i$  are determined by the asymptotic eigenvalues $e^{2\pi\ii\lambda_j/\ell}$ of the asymptotic holonomy of the instanton around the Taub-NUT isometry circle.
\item In addition, the interval $[\lambda_n-\ell,\lambda_1]$ is divided further into intervals $[\lambda_n-\ell,p_1], [p_1,p_2], ...[p_k,\lambda_1]$ by the addition of $k$ further points $\lambda_n-\ell<p_1<p_2<\ldots<p_k<\lambda_1$ (``$p$-points'') .  The location of these $p$-points  can be chosen (once and for all) as one wishes in the interval $(\lambda_n-\ell, \lambda_1)$.
 \item We then fix on each of the subintervals of the circle a unitary vector bundle. The rank of these varies.  On the interval $[\lambda_i, \lambda_{i+1}],i = 1,...n-1$, this bundle $V_i$ has rank $d_i= h^0(\bP^1, Q_i^{n-i})$; on $[\lambda_n-\ell,p_1]$, this bundle $V_n = V_{n,0}$ has rank $d_n=d_{n,0}=h^0(\bP^1, Q_n^{0})  $; on $[p_k,\lambda_1]$, the bundle $V_0$ has rank $d_0 = d_{n,k}=h^0(\bP^1, Q^n_{0}) $, and on $[p_j, p_{j+1} ], j = 1,..,k-1$, the bundle $V_{n,j}$ has rank $d_{n,j}= h^0(\bP^1,Q_{n,j}^0)$.
\item The bundles on successive intervals must be linked in some way at the boundary points. This is done by maps: at the $\lambda$-points $\lambda_i$, there are Hermitian injections $inj_i:E_{i-1}|_{\lambda_i}\rightarrow E_i|_{\lambda_i}$ or $inj_i:E_{i}|_{\lambda_i} \rightarrow E_{i-1}|_{\lambda_i}$ which identify  the  fibre on the lower rank side with a subspace of the fibre on the  higher rank side; under this identification, the adjoint map is a projection of the higher-dimensional fibre to its smaller-dimensional  subspace.  These maps are fixed. At the $p$-points, the maps are part of the data of the instanton; see below. \end{itemize}}

\subsubsection{Bow solutions}
 
 The instanton determines a {\em bow solution} of \cite{Cherkis:2008ip,Cherkis:2010bn} reviewed presently.  Its gauge equivalence class  corresponds  via the Nahm transform to our instantons over the Taub-NUT \cite{Second}.
We recall (adapting from \cite{Cherkis:2010bn}) the nature of the bow solution:
  
\begin{description}
\item[ A Nahm datum] is associated with each bow subinterval; it is a quadruplet $(\nabla_s, T_1,T_2,T_3)$ consisting of a unitary connection $\nabla_s$ and three skew Hermitian endomorphisms $T_1,T_2,$ and $T_3$ of the Hermitian bundle over that subinterval. Note that  the rank is varying with each change of interval.

\item[A bifundamental datum] is associated with each $p$-point, and  consists of   linear maps $M_j^L:V_{n,j-1}|_{p_j}\rightarrow V_{n,j}|_{p_j},\quad  M_j^R:V_{n,j }|_{p_j}\rightarrow V_{n,j-1}|_{p_j}$.
 
\item[A fundamental datum]   associated with each $\lambda$-point $s=\lambda_i$ is present only if the bundles have same rank $d_i= d_{i-1}$.  It is a pair  
$(I_i,J_i)$ of maps 
\begin{align*}
I_i:\, &\bC\rightarrow V_i|_{\lambda_i}=  V_{i-1}|_{\lambda_i},\\
J_i:\, &V_i|_{\lambda_i}=  V_{i-1}|_{\lambda_j}\rightarrow \bC. 
\end{align*}
\end{description}

These data satisfy equations:
\begin{enumerate}
\item
On the interior of the intervals, the Nahm data are smooth solutions of the Nahm equations, ordinary differential equations  discovered by Nahm in his early work on monopoles \cite{Nahm:1979yw}:  
$ \frac{dT_1}{ds} = [T_2,T_3],$ 
$\frac{dT_2}{ds} = [T_3,T_1],$ 
$\frac{dT_3}{ds} = [T_1,T_2].$
One can put in a gauge freedom by replacing the derivative $\frac{d}{ds}$ with a unitary  covariant  derivative
$\nabla$; the equations become 
\begin{equation}\label{Eq:Nahm} \nabla T_i =\sum_{j,k} \epsilon_{ijk} T_j T_k.\end{equation}

One can rewrite these equations with a spectral parameter, which is in fact the twistor parameter $\zeta\in\mathbb{P}^1.$ One has a Lax pair: 
\begin{align}\label{Laxpair}
A(\zeta,s) &= T_1 +\mathrm i T_2 -2 T_3\zeta - (T_1-\mathrm i T_2)\zeta^2,\\
D(\zeta,s) & =  -T_3- (T_1-\mathrm i  T_2)\zeta,
\end{align}
with the Lax equation
\begin{equation}\label{Laxpair2}[\nabla+D(\zeta, s),A(\zeta,s)]  = 0\end{equation}
being equivalent to the system of the Nahm equations \ref{Eq:Nahm}.

\item At $\lambda$-points the Nahm datum satisfies  matching conditions, which are the same as those which appear in Nahm's original work on monopoles. Suppose first that $d_{i-1}<d_i$, we ask that near $\lambda_i$, on the Hermitian bundle  $V_{i-1}$ of rank $d_{i-1}$  the solution $T^{i-1}_j$ to Nahm's equations on  $( \lambda_i -\epsilon, \lambda_i]$ be smooth up to the boundary.
 
On the Hermitian bundle $E_i$  of rank  $d_i$, over $[\lambda_i, \lambda_i+\epsilon)$, the solutions $T^{i}_j$ should be analytic on the interior, with the connection part of the solution smooth at the boundary. The matrices $T^{i}_j$, however  can develop a simple pole at $\lambda_i$. Decomposing in the vicinity of $\lambda_i$ the bundle $V_i$ into  the sum of $V_{i-1}|_{\lambda_i}\times[\lambda_i,\lambda_i+\epsilon)$ and its orthogonal complement, 
 
\begin{equation*}
T_j(s)=\begin{pmatrix}a_j(s)                & b_j( s)\\
                        c_j(s)&-\frac{\rho_j}{2(s-\lambda_i)}+d_j(s)                 \end{pmatrix}.
                         \end{equation*}
Here, the top left block is $d_{i-1}\times d_{i-1}$, the bottom right block is
$m\times m$, with $m= d_i-d_{i-1}$; we ask that $a_j,b_j,c_j, d_j$ be analytic at $s=\lambda_i$, and $\{\rho_j\}_{j=1}^3$ be the  components of the $m$-dimensional irreducible representation  of $su(2)$ in its standard basis, that is the standard representation in $m$ dimensions of the Pauli matrices.
Furthermore, the  solutions on the two intervals should  match   by
\begin{equation*}
a_j(\lambda_i) = T^{i-1}_j(\lambda_i).
\end{equation*}
In the same way, the connection coefficients match also.  See \cite{Second} for details.
 
 \bigskip
\noindent {\it When $d_{i-1}>d_i$:}
 
 One has the same boundary conditions, with the roles of the two intervals and corresponding bundles reversed.
 
 \bigskip
\noindent {\it When $d_{i-1}=d_i$:}
  We have an identification of the fibres from both sides. 
 One asks that the solutions $A^{i-1}, A^i$ have respective finite limits at  $\lambda_i$, with a difference that is of rank one:

\begin{equation}\label{lambdapoint}
A^i(\zeta,\lambda_i)-A^{i-1}(\zeta,\lambda_i) = (I_i -J_i^\dagger\zeta)(J_i +I_i^\dagger\zeta).
\end{equation}
  The connection is continuous at the boundary under the identification.

\item At $p$-points the matrices $A^{n,j-1}(\zeta,p_j),A^{n,j }(\zeta, p_j)$ are required to satisfy:
\begin{align}\label{ppoint}
A^{n,j-1}(\zeta,p_j)=& (M_j^R+\zeta (M_j^L)^\dagger)(M_j^L-\zeta(M_j^R)^\dagger) + z_j(\zeta),\\ 
A^{n,j }(\zeta, p_j) =& ( M_j^L-\zeta (M_j^R)^\dagger)(M_j^R +\zeta (M_j^L)^\dagger)  + z_j(\zeta)\end{align}
 
Here $z_j(\zeta) = (p^1_j-\ii p^2_j) -2 p^3_j \zeta - (p^1_j-\ii p^2_j)\zeta^2$ encodes the position $(p^1_j,p^2_j, p^3_j)$ of the $j$-th NUT.
 
\end{enumerate}

The gauge equivalence class of each bow solution $(\nabla, T_1,T_2,T_3, B^L,B^R,I,J)$ corresponds to an instanton on TN$_k$ with NUTs at $\vec{p}_j$ and its asymptotic holonomy eigenvalues $\exp(2\pi\ii\lambda_j/\ell).$ 

\subsection{Bow complexes}\label{Sec:BowClxs}

The Nahm equations are reductions to one dimension of the anti-self-duality equations in four dimensions. As is usual for these equations, the Nahm equations admit a version of the Kobayashi-Hitchin correspondence, brought to light by Donaldson in \cite{Donaldson:1985id}.  In this context the correspondence involves splitting the equations into two pieces, one complex, and the other real; the relevant boundary conditions likewise get split into complex and   real parts.  The complex equation and its boundary conditions determines a holomorphic object and corresponds to fixing one complex structure in a hyperk\"ahler family; the real equation is variational in nature, and is   the Euler-Lagrange equation for  an appropriate hermitian metric. The correspondence tells us that to each (stable) solution to the complex equations (with boundary conditions), there is a unique solution to the real equations (with boundary conditions), obtained by choosing appropriate metrics.  {Thus, to describe moduli, one only needs to consider stable solutions to the complex  equations (with boundary conditions);  in general for these moduli problems, these solutions with their boundary conditions are called Nahm complexes, and here  in this particular context are referred to as bow complexes.} 

In our case, the complex data is obtained by choosing to restrict our parameter $\zeta$ to zero in (\ref{Laxpair}),(\ref{Laxpair2}),  (\ref{lambdapoint}), (\ref{ppoint}). One has 
\begin{itemize}
\item On the interior of the  intervals, sections $A^i(0,z) = T_1(z) +\ii T_2(z), A^{n,j}(0,z)$ of the bundles $End(E_i), End(E_{n,j})$, and a complex connection $\nabla_z = \frac{d}{dz} + D(0,z)$, with $\nabla_z A(0,z) = 0;$
\item At the $\lambda$-points, matching conditions given by specializing the Nahm conditions to $\zeta = 0$;
\item At the $p$-points, again specializing to $\zeta = 0$:
\begin{align}\label{ppoint2}
A^{n,j-1}(0,p_j)=&  M_j^R M_j^L + z_j(0),\\ 
A^{n,j }(0, p_j) =&  M_j^L M_j^R  + z_j(0).
\end{align}

\end{itemize}
\begin{proposition} An instanton bundle encodes a bow complex. In terms of the sheaves $P,Q, R$
\begin{itemize}
\item On the interval $[\lambda_i, \lambda_{i+1}]$, the rank $d_i$ bundle is given as $[\lambda_i, \lambda_{i+1}]\times H^0(\bP^1, Q_i^{n-i})$, with the trivial connection. The endomorphism $A^i(0,z)$ is given by the action of multiplication by $\eta$ on $H^0(\bP^1, Q_i^{n-i})$, and so is given by the matrix $\beta_i$ of (\ref{resol-P-Q}) . Likewise, on the intervals
$[\lambda_n-\ell,p_1],$ $[p_1,p_2], ...[p_k,\lambda_1]$, we have constant fibres $H^0(\bP^1, Q_{n}^ 0),$ $H^0(\bP^1,$ $Q_{n,1}^ 0),$ $\ldots,$ $H^0(\bP^1, Q_{n,k}^ 0)$, of ranks $d_n,d_{n,1},..,d_{n,k}= d_0$ with covariant constant endomorphisms given by the matrices $\beta_n, \beta_{n,1},...,\beta_{n,k}= \beta_0$ of the resolutions in (\ref{matrices}, \ref{matricesQnj} ).
\item At the $\lambda$ points, the identifications at the boundary  are ensured by the maps of sheaves $Q_i^{n-i}\leftarrow P_i^{n-i-1}\rightarrow Q_{i+1}^{n-i-1}$, and the induced maps on global sections.
\item Likewise at the $p$-points, the  correspondences emerge from the sheaf maps $Q_{n,j-1}^{0}\leftarrow R_j\rightarrow Q_{n,j}^{0}$ and the induced maps on global sections. This identifies  $M_j^R$ as $M_{\xi_j}$, and $M_j^L$ as $M_{\xi_j}$;  referring to   (\ref{matrices}), this gives, once one takes resolutions:
 \begin{align} 
 \beta_{n,j} &= M_{\xi_j}M_{\psi_j} + z_j,&
  \beta_{n,j-1} &= M_{\psi_j}M_{\xi_j} + z_j.
 \end{align}
\end{itemize}
\end{proposition}
  {The proof is a minor modification of what is already found in Section 3 of  \cite{Hurtubise:1989wh}, 
Section 4 of  \cite{Hurtubise:1989qy}, Section 5 of \cite{Charbonneau:2006gu}; in all of these cases, the holomorphic instanton or monopole bundle is described in term of simple sheaves, analogous to $\pP, \pQ$, and the Nahm complex data is extracted in the same way. It is perhaps worth briefly explaining where the correspondence with the holomorphic data comes from; the picture is quite general, and holds for monopoles, calorons, and here in the Taub-NUT  background. The solutions to Nahm's equations associated to the instanton are given, for a fixed $t$ in the relevant interval,  as endomorphisms of the space of sections $V^i_t$ of what is generically a line bundle $L_t$ over spectral curves $S_i$ sitting in $T\bP^1$; the space $V^i_t$ of sections is intrinsically the set of $L^2$ solutions to a shifted Dirac operator (again, $t$ parametrizes the shift) , and gets identified with global sections of $L_t$ over the $S_i$.  The sheaf $Q_i^{n-i}$ is essentially the restriction of this line bundle $L_t$ to the intersection of   a fibre of $T\bP^1\rightarrow \bP^1$ with $S_i$ ( this intersection is a sum of points), i.e a restriction of the twistor data to one complex structure; we have, by restriction to this fibre, that $V^i_t = H^0(\bP^1,Q_i^{n-i})$.  The sheaves $P_i^{n-i-1}$ also correspond to global sheaves on $T\bP^1$, and once restricted to a fiber, map to both $Q_i^{n-i}$ and $Q_{i+1}^{n-i-1}$, provide the links between the bundles on adjacent intervals, as noted in the statement of the Proposition.}

While in the preceding proposition, the data for the bow complex is given in terms of sheaves, more concretely, as is already apparent, the essential data for a Nahm complex arises from the matrices obtained by resolving the sheaves. These matrices will satisfy certain relations, typically arising from the commutation of the diagrams of resolutions. Thus, we have:

\begin{itemize}
\item Complex matrices $\beta_i$, of size $d_i\times d_i$, $i= 0,..n,$
\item Complex matrices $\beta_{n,j}$ of size $d_{n, j}\times d_{n,j}$, $j = 0,..k$, with the convention  $\beta_{n,0} = \beta_n, \beta_{n,k} = \beta_n,$
\item Complex matrices $A_i$ of size $d_{i+1}\times d_i$, $ i= 0,...,n-1$,
\item Complex matrices $\alpha_i$ of size $d_{i+1}\times 1$, $ i= 0,...,n-1$,
\item Complex matrices $\gamma_i$ of size $1\times d_{i+1} $, $ i= 0,...,n-1$,
\item Complex matrices $M_{\xi_j}, M_{\psi_j}$ of size $ d_{n,j }\times d_{n,j-1}, d_{n,j-1}\times d_{n,j} $, $ j= 0,...,k$,
\end{itemize}
with relations:

\begin{align}
 0= &\beta_{i+1} A_i - A_i \beta_i - \alpha_i\gamma_i,& i&=0,..,n-1,  \\
\beta_{n,j} = & M_{\xi_j}M_{\psi_j} + z_j,\\
 \beta_{n,j-1} =& M_{\psi_j}M_{\xi_j} + z_j,
\end{align}
and in addition the genericity condition, that the sequences
\begin{equation}\label{exact-p-q2}0\rightarrow \pO(-1)^{d_{i }} \xrightarrow{\begin{pmatrix}\eta-\beta_i\\ -\gamma_i\\-A_i \end{pmatrix}}  \pO^{d_{i}+1}\oplus \pO^{d_{i+1}}(-1)\xrightarrow{\begin{pmatrix}A_i, \alpha_i, \eta-\beta_{i+1}   \end{pmatrix}} \pO^{d_{i+1}}\rightarrow 0 
\end{equation}
 be exact (pointwise) for all $\eta$.
 
 This data, plus a stability condition, is (up to a shift by scalars in the $\beta_i$), is exactly the quiver data of Nakajima and Takayama \cite{Nakajima:2016guo} for a bow solution of Nahm's equations. The stability condition is given in \cite[Sec.2.4]{Nakajima:2016guo}; it is essentially an irreducibility condition, saying that there are no sub-objects or quotient objects. With this, 
 \begin{proposition} (Nakajima and Takayama) The data above determines a bow solution.\end{proposition}

\section{$SO$ and $Sp$-bundles}

Let us now consider the cases of $SO(n)$ and $Sp(n/2)$ instantons; here the $n$  denotes the rank of the vector bundle in the standard representation, so that $n$ is even in the $Sp$ case. Our instantons are $U(n)$ instantons, equipped with extra structure; a trivial top exterior product and a pairing, either symmetric or skew;  and of course a connection compatible with these structures. The eigenvalues of the asymptotic holonomy come in pairs, that are inverses of one another; taking logs, one has $\lambda_i = -\lambda_{n-i+1}$. The normalisation we then use for these eigenvalues is 
$$-\ell/2 <\lambda_1 <\lambda_2<...< (\lambda_{n-1}= -\lambda_2)  <   (\lambda_{n}= -\lambda_1) <\ell/2.$$
With this ordering, the asymptotic form of the pairing is anti-diagonal.

The compactification process then defines a bundle $E$ on $X$, as for the $U(n)$ case. The fact that the instanton has trivial determinant tells us that $E$ also has trivial determinant. In addition, the flags $E_{0,i}$ and $E^{\infty,i}$ that we obtain as for the $U(n)$ case along $C_0, C_\infty$ are isotropic-coisotropic; that is the first half of the spaces are isotropic, and the remaining ones are the annihilators of the first ones. 

For the various degrees, the triviality of the determinant and the isotropy-coisotropy imply:
\begin{align}
n_i=&0,\\
m_i =& -m_{n-i+1},\\
d_i = & d_{n-i}.
\end{align}

With that, we can go through the same procedure as for $U(n)$ and obtain sheaves $Q_i^{n-i} , P_i^{n-i-1} $. When one puts in the pairings, we will see that we will have three equivalent sets of data.
 \begin{proposition}
 Let $E$ be an $SO(n)$ (alternately, $Sp(n/2)$) instanton. Then one has equivalent data:
\begin{itemize} 
\item{} The   symmetric  (alt. skew-symmetric)  holomorphic pairing on $E$;
\item{} Suitable meromorphic pairings $<,>_i: P^{ n-i -1}_{i }\otimes P^{i}_{n-i-1}\rightarrow {\pO}$, with poles over the support of the sheaves $Q$, with $<,>_i = <,>_{n-i}$ ( alt. $<,>_i = -  <,>_{n-i}$)
\item{} Suitable pairings
$\{,\}_i: H^0(\bP^1, Q_i^{n-i}) \otimes H^0(\bP^1, Q^i_{n-i})\rightarrow \bC$ with $\{,\}_i = -  \{,\}_{n-i}$ ( alt. $\{,\}_i =   \{,\}_{n-i}$)
\end{itemize}
The formulae linking these pairings are given in (\ref{EtoP}), (\ref{EtoQ}) and (\ref{PtoQ2}), below.
\end{proposition}

Note the change on the parity of the pairings as one goes from the $P^{ n-i -1}_{i }$ to the $ Q_i^{n-i}$. Let us start with a pairing on $E$: this yields a meromorphic pairing on $\oplus \pi^* P_i$.  One has the exact sequence (schematically)
$$0\rightarrow  E \rightarrow \oplus \pi^* P^{i}_{n-i-1} \oplus R \rightarrow \oplus \pi^* Q^i_{n-i}\rightarrow 0,$$
where the map $ E\rightarrow \oplus P^{i}_{n-i-1}$ is an isomorphism 
away from the support of the $Q^i_{n-i}$. One can then push down the pairing to the $P$, away from the support of $Q$; at the support, the pairing will have poles.

In a more explicit fashion, one has the inclusion of sheaves, with the diagonals giving short exact sequences:
  \begin{equation} \label{sequence E}\xymatrix{& &&\pO|_{C_\infty}&\\  && E_{i}^{n-i }\ar[ur]\ar[dr]&&\pO|_{C_\infty}\\
  &E_{i }^{ n-i-1}\ar[ur]\ar[dr]&&E_{i+1}^{n-i }\ar[ur]\ar[dr]&\\  &&E_{i+1 }^{n-i -1}\ar[ur]\ar[dr]&&\pO(m_{i+1})|_{C_0}\\
  & &&\pO(m_{i+1})|_{C_0}} \end{equation}
  
   Taking direct images, one has that $\pi_*(E_{i+1}^{n-i })= S_{i+1}^{n-i }$ is a rank one sheaf. There is then a mapping 
$$\phi_{i}:S_{i+1}^{n-i }\rightarrow P_{i}^{n-i-1},$$
which is an isomorphism away from the support of the $Q$. Since the  $S_{i+1}^{n-i }$ are subsheaves of $\pi_*(E)$, one can apply the pairing to them.
Note that since $ S_{i+1}^{n-i }$ takes values in $E_{0,i+1}$ over $C_0$, and $E^{\infty,n-i}$ over $C_\infty$, and one can evaluate the pairing at any point in the fibre, the only non-zero pairings are between $ S_{i+1}^{n-i }$ and $ S^{i+1}_{n-i }$; it follows that the only non-zero pairings for the $P$ are between $ P_{i }^{n-i-1 }$ and $ P^{i }_{n-i -1}$. More invariantly, using the pairing on $E$, there is a natural pairing 
$$(\cdot,\cdot)_{i} : [\pi_*(E_{i+1}^{n-i })= S_{i+1}^{n-i } ]\otimes [R^1\pi_*(E^{i }_{n-i -1}) = P^{i }_{n-i -1}]\rightarrow R^1\pi_*(\pO(-C_0-C_\infty))\simeq \pO$$
Explicitly, on cocycles, this defines our pairing $<\cdot, \cdot>_{i}$ between $ P_{i }^{n-i-1 }$ and $ P^{i }_{n-i -1}$ by
\begin{equation}\label{EtoP}<s, t>_{i} =  \mathop{Res}_\xi((\phi_i^{-1}(s) , t)) \frac{d\xi}{\xi}.\end{equation}
 One has that $<s, t>_{i} = <t,s>_{n-i-1}$ if the pairing on $E$ is symmetric, and $<s, t>_{i} = - <t,s>_{n-i-1}$ if the pairing on $E$ is skew.

One can use the  resolutions of the sheaves $P, Q$.  Recall that we have exact sequences:

\begin{equation} \label{resolution-PQ}\xymatrix{   
 & \pO(-1)^{d_i} \ar[r]^{Z_{i} } &\pO^{d_i}\ar[rr]& &Q^{n-i}_{i}\\
 \pO(-1)^{d_i}\ \  \ar[dr]^{\tilde X_{i} }\ar[r]^{W_{i}}\ar[ur]^{X_{i} }& \pO^{d_i+1} \ar[dr]^{\tilde Y_{i} }\ar[ur]^{Y_{i} }\ar[rr]&& P^{n-i-1}_{i} \ar[dr] \ar[ur]  \\ 
 & \pO(-1)^{d_{i+1} }\ar[r]^{Z_{i+1}} &\pO^{d_{i+1}}\ar[rr]& &Q^{n-i-1}_{i+1}\\
}
\end{equation}
with

\begin{equation} 
\begin{matrix}
Z_{i}&=&\begin{pmatrix}\eta-\beta_{i}\end{pmatrix},&&Z_{i+1 } &=&(\eta-\beta_{i+1}) \\ \ 
\tilde X_{i }&=& \begin{pmatrix}\alpha_{i}\end{pmatrix},&& \tilde Y_{j }&=& \begin{pmatrix}\alpha_{i}, \alpha'_{i}\end{pmatrix},\\ \\
X_{i}&=& \begin{pmatrix}1\end{pmatrix},&&Y_{i}&=& \begin{pmatrix}1&0 \end{pmatrix},\\ \\
W_{i}&=&  \begin{pmatrix}\eta-\beta_{i} \\ - \gamma_i  \end{pmatrix},&& \\ \\
\end{matrix}\label{matrices-caloron}
\end{equation}

For the $n-i-1$ case, we have
\begin{equation} \label{resolution-PQ2}\xymatrix{   
 & \pO(-1)^{d_{i+1}} \ar[r]^{Z_{n-i-1}} &\pO^{d_{i+1}}\ar[rr]& &Q_{n-i-1 }^{i+1}\\
 \pO(-1)^{d_{i+1}}\ \  \ar[dr]^{\tilde X _{n-i-1}}\ar[r]^{W_{n-i-1}}\ar[ur]^{X_{n-i-1}}& \pO^{d_{i+1}+1} \ar[dr]^{\tilde Y_{n-i-1} }\ar[ur]^{Y_{n-i-1} }\ar[rr]&& P_{{n-i-1}}^{i} \ar[dr] \ar[ur]  \\ 
 & \pO(-1)^{d_{i}} \ar[r]^{Z_{n-i}} &\pO^{d_{i}}\ar[rr]& &Q_{n-i}^{i}\\
}
\end{equation}

\begin{equation} 
\begin{matrix}
Z_{n-i-1}&=&\begin{pmatrix}\eta-\beta_{n-i-1}\end{pmatrix},&&Z_{n-i} &=&(\eta-\beta_{n-i}) \\ \\
X_{n-i-1}&=& \begin{pmatrix}1 \end{pmatrix},&& Y_{n-i-1}&=& \begin{pmatrix}1&0 \end{pmatrix},\\ \\
\tilde X_{n-i-1}&=& \begin{pmatrix}\alpha_{n-i-1}\end{pmatrix},&&\tilde Y_{n-i-1}&=& \begin{pmatrix}\alpha_{n-i-1}&\alpha'_{n-i-1}\end{pmatrix},\\ \\
&&&&W_{n-i-1}&=& \begin{pmatrix}\eta-\beta_{n-i-1}\\ -\gamma_{n-i-1}
 \end{pmatrix}\\ \\
\end{matrix}\label{matrices-caloron2}
\end{equation}

Here we have split the matrices into appropriate blocks. There are relations among the matrices imposed by the commutativity of the diagrams:
\begin{align}  
\begin{pmatrix}\alpha_{i}&\alpha'_{i}\end{pmatrix}\begin{pmatrix} \beta_{i} \\ \gamma_{i}  \end{pmatrix}&= \begin{pmatrix} \beta_{i+1}\end{pmatrix} \begin{pmatrix}\alpha_{{i}}\end{pmatrix} \\
 \begin{pmatrix}\alpha_{n-i-1}&\alpha'_{n-i-1}\end{pmatrix} \begin{pmatrix} \beta_{n-i-1}\\ \gamma_{n-i-1}
 \end{pmatrix}&= \begin{pmatrix} \beta_{n-i}\end{pmatrix} \begin{pmatrix}\alpha_{n-i-1}\end{pmatrix}
 \end{align}
 
 Lifting the pairings on the $P$ to   pairings $<\cdot,\cdot>$ on their spaces of sections, one wants $<W_{i}u, t> = <s,W_{n-i-1}v>=0$ for all $u,v,s,t$. As both $W_{i}, W_{n-i-1}$ are  of maximal rank for $\eta$ generic, the matrix of the pairing must be of rank one, and so of the form $U \cdot V^T$, for    column vectors $U,V$, with $W_{i}^TU = 0,  (W_{n-i-1})^T V= 0$. Write $U = \begin{pmatrix}u\\ u'\end{pmatrix}$, with $u, u'$ of size $m_{i}, 1$ respectively. We then have 
$$W_{i}^TU= (\eta-\beta_{i})^T u -\gamma_{i}^T u' .$$
Solving, we find, up to scale,  
$$u = ( (\eta-\beta_{i})^{-1})^T \gamma_{i}^T , \quad u'  = 1 ;$$
 likewise,   
we get 
$$v =  ((\eta-\beta_{n-i-1} )^{-1})^T \gamma_{n-i-1}^T  ,\quad v'  = 1  ,$$
again up to scale.  The scalar product of two vectors $A = \begin{pmatrix} a\\  a'\end{pmatrix}, B= \begin{pmatrix}  b\\   b'\end{pmatrix}$ is then the product of scalars 
\begin{align*}
<A,B>_{i} =  &(a^Tu +a'u')( b^Tv + b'v') f\\
 =&[ (a^T  (  (\eta-\beta_{i})^{-1})^T \gamma_{i}^T) + a' ]\\
&\cdot [ (b^T ((\eta-\beta_{n-i-1} )^{-1})^T \gamma_{n-i-1}^T) + b' ] \cdot f\\
=&[  \gamma_{i} (\eta-\beta_{i})^{-1} a + a' ]\\
&\cdot [ \gamma_{n-i-1} (\eta-\beta_{n-i-1} )^{-1}b + b' ] \cdot f,
\end{align*}
for some function $f$. Note, that at infinity, this becomes $a'b'f$, and so is finite as long as $f$ is; furthermore if $A $ maps to zero in $Q_{i}^{n-i}$,  and $B$ maps to zero in $Q^{n-i-1}_{i+1}$, so that 
$a = (\eta-\beta_{i})\mu,  b  = (\eta-\beta_{n-i-a})\nu$ for some $\mu, \nu$, the scalar product is again finite if $f$ is at that point. Asking for this finiteness, and the fact that $f$ must be finite elsewhere tells us that $f$ is a constant $f_{i}$.  Thus  the pairing   will be fixed by the value at infinity; we can set
\begin{align} f_{i} &= 1\ ( \mathrm {SO \ case}),\\
f_{i} &=
\begin{cases} 1,&\text{for }\ i<N/2 \\
-1,&\text{for } i\geq N/2
\end{cases}
( \mathrm {Sp \ case}).
\end{align}
Notice that with these definitions the forms have the correct symmetry, between $<,>_i$ and $<,>_{n-i-1}$.

Pairings on the spaces of sections of the $Q^{n-i}_{i}$ can be defined as follows. Recall that 
\begin{align}
Q_{n-i}^{i} =& R^1\pi_*(E_{n-i}^{i})= R^1\pi_*(E_{n-i}^{i}(-F))\\
Q^{n-i}_{i} =& R^1\pi_*(E^{n-i}_{i})= R^1\pi_*(E^{n-i}_{i}(-F))
\end{align}
(The support of the $Q$'s lieis away from $\eta = \infty$,  and so we can twist by $F$); this yields
\begin{align}
H^0(\bP^1, Q_{n-i}^{i}) =& H^1(X,E_{n-i}^{i}(-F)),\\
H^0(\bP^1, Q^{n-i}_{i}) =& H^1(X,E^{n-i}_{i}(-F)).
\end{align}
On the other hand, there is a pairing 
$$E^{n-i}_{i}\otimes  E_{n-i}^{i}\rightarrow \pO(-C_0 -C_\infty);$$
and so an induced pairing 
\begin{equation}\label{EtoQ} \{\cdot,\cdot\}_j: H^1(X,E^{n-i}_{i}(-F))\otimes H^1(X,E_{n-i}^{i}(-F)) \rightarrow H^2(X,\pO(-C_0 -C_\infty-2F)).\end{equation}
The latter is dual to $H^0(X, K_X(C_0+C_\infty+ 2F) )  $, and is one dimensional. The pairing is then perfect (essentially Serre duality) and localises (one can represent the classes by forms supported on a neighbourhood of the pull backs of the support of the $Q$, so that  if two sections of $Q_{n-i}^{i} ,Q^{n-i}_{i} $ have disjoint support, they pair to zero).  Also, from the skew symmetry of the cup product,  a symmetric pairing on $E$ gives skew symmetric pairings on the sections of $Q$, and a skew symmetric pairing on $E$ gives symmetric pairings on sections of $Q$.

\begin{lemma}\label{adjoint}
 
  With respect to the pairing $\{,\}_j$, the maps $\beta_i,\beta_{n-i}$ are  adjoint: $\beta_i = \beta_{n-i}^*$.
Thus, letting the pairing  $  \{,\}_{i}$ have  as matrix $ K_i$ in  a  basis,
  $\beta_i^TK_i  = K_i \beta_{n-i}$.
  
In the same vein the maps 
\begin{align}M_\psi: H^0(\bP^1, Q_0^n)&\rightarrow H^0(\bP^1, Q^0_n) ,\\ M_\xi: H^0(\bP^1, Q^0_n )&\rightarrow H^0(\bP^1, Q_0^n)\end{align} are self adjoint, in the sense that for any sections $u,u'$ in $H^0(\bP^1, Q_0^n)$, and $v,v'$ in $H^0(\bP^1, Q^0_n)$
$$\{M_\psi(u), u'\}_0 = \{u, M_\psi(u')\}_n, \qquad \{M_\xi(v), v'\}_n = \{v, M_\xi(v')\}_0$$
\end{lemma}

 \begin{proof} The maps  $\beta_i,\beta_{n-i}$ are the effect of multiplication by the coordinate $\eta$ on global sections, and one has that 
 $\{\eta t_i, t_{n-i }\}_i  = \{ t_i, \eta t_{n-i }\}_j$. Explicitly, $\beta_i^TK_i  = K_i \beta_{n-i}$. Likewise, the maps $M_\psi, M_\xi$ are the effects on the direct image of multiplication by $\psi, \xi$ respectively. \end{proof}

Explicitly,  if $s$ is a section of $Q^{n-i}_{i}$, one has from  the resolution of $Q^{n-i}_{i}$ that $(\eta-\beta_i)s = 0$ as a section of $Q^{n-i}_{i}$. This means that on a disc $U_\alpha$ surrounding a point of $p_\alpha$ of the support, if $s$ is represented by a cocycle $\sigma = \sigma_{0\infty}$ with respect to the covering $ \pi^{-1}(U_\alpha)\cap (X-C_\infty), \pi^{-1}(U_\alpha)\cap (X-C_0)$, then one can write the cocycle $(\eta-\beta_i)s$ as a coboundary 
$f_0 -f_\infty$; doing so for all the union $U$ of the discs $U_\alpha$ allows us to write
$$\sigma = (\eta-\beta_i)^{-1}(f_0-f_\infty)$$
If now $\tau$ is the cocycle representing a section $t$ of $Q_{n-i}^{i}$, the explicit form for a two-cocycle with respect to the open sets  $\pi^{-1}(U), X-C_0, X-C_\infty$ representing the pairing  in $ H^2(X,\pO(-C_0 -C_\infty-2F))$ is given by 
$$  < (\eta-\beta_i)^{-1}f_0, \tau>_E,$$
or by
$$ - < (\eta-\beta_i)^{-1}f_\infty, \tau>_E.$$
In the generic case, the section $f_0$ can be chosen to be a section of $S_i^{n-i+1}$.
The explicit number coming out of the pairing is then obtained by representing the cocycle as a form 
$$ < (\eta-\beta_i)^{-1}f_0, \tau>_E  \frac{d\xi}{\xi}\wedge d\eta,$$
 and taking the residues both in $\xi$ and $\eta$, and summing over the $\alpha$ in the support of   $Q^{n-i}_{i}$, so that 
\begin{align*}
\{s,t\}_{i} =&   \mathop{Res}_{\eta=\infty} \mathop{Res}_{\xi = 0} \left(< (\eta-\beta_i)^{-1}f_0, \tau>_E  \frac{d\xi}{\xi}\wedge d\eta\right)\\
 = &-\mathop{Res}_{\eta=\infty} \mathop{Res}_{\xi = 0} \left(< (\eta-\beta_i)^{-1}f_\infty, \tau>_E   \frac{d\xi}{\xi} \wedge d\eta\right).
\end{align*}

Our next aim is to relate the pairing on the $P$s to those on the $Q$s. Let $s$ be a section of $P^{n-i-1}_{i}$; it gives sections $\pi_{i} (s) $ of $ Q^{n-i }_{i}$, and $\tilde \pi_{i} (s)$ of $Q^{n-i-1}_{i+1}$; likewise, a section $t$ of $P_{n-i-1}^{i}$ gives sections $\pi_{n-i-1} (t) $ of $ Q_{n-i -1}^{i+1}$, and $\tilde \pi_{n-i-1} (t)$ of $Q^{i}_{n-i}$.

\begin{lemma}\label{PtoQ2} 
 We have
\begin{align*}<s,t>_{i} =& \{(\eta-\beta_{i})^{-1}\pi_{i} (s), \tilde \pi_{n-i-1} (t)\}_{i} \\&- \{(\eta-\beta_{i+1})^{-1}\tilde \pi_{i} (s), \pi_{n-i-1} (t)\}_{i+1} \\&+
 <s(\infty),t(\infty)>_E\end{align*}
 \end{lemma}  
\begin{proof} Let us place ourselves in the generic situation in which the eigenvalues of $\beta_{i}, \beta_{i+1}$ are simple, and disjoint; the general case will follow by continuity. The scalar product $<s,t>_{i}$ is a meromorphic function, with $d_i+ d_{i+1} $ poles at union of the supports  of $ Q^{n-i}_{i}$, $Q^{n-i-1}_{i+1}$, which lie away from $\eta =\infty$. We note that such a function is determined by its residues at the poles, plus its value at infinity.

Now represent $s, t$ by cocycles $\sigma, \tau$ as in the preceding lemma; note that $\sigma, \tau$ also are representative cocycles for the projections $\pi_{i} (s),$ $\tilde \pi_{i} (s),$ $\pi_{n-i-1} (t),$ $\tilde  \pi_{n-i-1} (t)$. To define the pairing on the $P$, we had represented $\sigma$ as a  section $f = \phi^{-1}(s)$ of   $S_{i+1}^{n-i }$
 away from the support of $Q$, with simple poles at these points, and set $<s, t>_{i} = \mathop{Res}_{\xi=0}((f , t)_{i}) \frac{d\xi}{\xi}.$ 
 
Now let us look at the pairing of $ Q^{n-i}_{i}$ with $ Q_{n-i}^{i}$. Let us consider an eigenbasis $s_\mu$ of $\beta_i$ acting on $H^0(\bP^1, Q^{n-i}_{i})$ with $s_\mu$ supported at distinct points $ z_\mu$; in this basis, $\beta_{i}= \diag_\mu( z_\mu)$ . Decomposing $(\eta- z_\mu) s_\mu$ as $f_{\mu,0} -f_{\mu,\infty}$ as in the previous lemma, the pairing of $s_\mu$ with $t$ becomes 
  $$ \{s_\mu, t\}_{i} =   \mathop{Res}_{\eta= z_\mu} \mathop{Res}_{\xi = 0} \left(< (\eta- z_\mu)^{-1}f_{\mu,0}, \tau>_E  \frac{d\xi}{\xi}\wedge d\eta\right) ,$$
  and so 
  $$\{(\eta-\beta_{i})^{-1}s_\mu, t\}_{i} = (\eta- z_\mu)^{-1} \mathop{Res}_{\xi= 0} \left(< f_{\mu,0}, \tau>_E \frac{d\xi}{\xi} \right)   .$$
 If then $s = \sum_\mu a_\mu s_\mu$, so that $f_0 =  \sum_\mu a_\mu f_{0,\mu}$
  $$\{(\eta-\beta_{i})^{-1}s , t\}_{i} =\sum_\mu   (\eta- z_\mu)^{-1} \mathop{Res}_{\xi= 0} \left(<f_0, \tau>_E \frac{d\xi}{\xi} \right)   .$$
and so $\{(\eta-\beta_{i})^{-1}s , t\}_{i}$ and  $<s,t>_{i}$ have the same residue in $\eta$ at $ z_\mu$.
  If one looks at the pairing of $ Q^{n-i-1 }_{i+1 }$ with $ Q_{n-i-1 }^{i+1}$, one has the same formula, but with a different sign, as $f$ now represents the sum $\sum_\mu a_\mu f_{\mu,\infty}.$
 
 Our residues now match; all that remains is the value at infinity, which is immediate.
 \end{proof}
Now, let us examine some of the properties of the pairings, and the consequences for the maps in the resolutions of the $P,Q$ (recalling the definitions of the matrices in (\ref{resolution-PQ}), (\ref{resolution-PQ2}).
 
 \begin{lemma}\label{adjoint2}
 
 \begin{itemize}
  
 \item Taking residues at infinity in the preceding proposition, 
$$ \mathop{Res}_\infty( <s,t>_{i}) = \{\pi_{i} (s),\tilde \pi_{n-i-1}(t)\}_{i}  - \{\tilde\pi_{i}(s), \pi_{n-i-1}(t)\}_{i+1} $$

\item In consequence, letting the pairings $\{,\}_{i}, \{,\}_{i+1}$ have matrices $K_{i}, K_{i+1}$ in our bases,
  \begin{align*} 
K_{i}  \alpha_{n-i-1}- \alpha_{i}^T K_{i+1 } &=0 \\
 K_{i}  \alpha_{n-i-1}' & =\gamma_{i}^Tf_{i}\\
-(\alpha_{i}')^TK_{i+1}&=\gamma_{n-i-1}f_{i}.\end{align*}

\item Viewing the maps $\alpha_{i}, \alpha_{n-i-1}$ as maps $H^0(\bP^1, Q^{n-i}_{i}) \rightarrow H^0(\bP^1, Q^{n-i-1}_{i+1}) $, $H^0(\bP^1, Q_{n-i-1}^{i+1}) \rightarrow H^0(\bP^1, Q_{n-i}^{i }) $ via the isomorphisms \newline
$H^0(\bP^1, P^{n-i-1 }_{i}(-1))\rightarrow H^0(\bP^1, Q^{n-i}_{i}),$ $H^0(\bP^1, P_{n-i-1}^{i }(-1)) \rightarrow H^0(\bP^1, Q_{n-i-1}^{i+1})$, the maps $\alpha_{i}, \alpha_{n-i-1}$ are adjoints of each other.

\item In the same way, taking a trivial scalar product on a one-dimensional vector space, the maps $M_{i}  \alpha_{n-i}' ,  \gamma_{i}^Tf_{i}$ are adjoints of each other, as are the maps $-(\alpha_{i}')^TM_{i+1},  \gamma_{n-i-1}f_{i}$.
 \end{itemize}
 \end{lemma}
 
 \begin{proof}  
  
  Writing the sections $s,t$ as
 $$ s = \begin{pmatrix}\tilde s \\ s '\end{pmatrix},\quad t = \begin{pmatrix}\tilde t \\t'\end{pmatrix},$$
 and letting the pairings $\{,\}_{i}, \{,\}_{i+1}$ have matrices $K_i, K_{i+1}$ in our bases, we have 
 \begin{multline} 
  \tilde s^T K_{i} \begin{pmatrix}\alpha_{n-i-1}& \alpha_{n-i-1}' \end{pmatrix}\begin{pmatrix}\tilde t\\t'\end{pmatrix} -
 \begin{pmatrix}\tilde s^T& s' \end{pmatrix}\begin{pmatrix}\alpha_{i}^T\\(\alpha_{i}')^T\end{pmatrix}K_{i+1} \tilde t \\
 = \mathop{Res}_\infty[ ( \tilde s^T (  (\eta-\beta_{i})^{-1})^T \gamma_{i}^T  + s' ) ( \gamma_{n-i-1} (\eta-\beta_{n-i-1} )^{-1} \tilde t +t')]f_{i}.
  \end{multline}
  If $s ' = t'=0$, one finds 
  $$K_{i}  \alpha_{n-i-1} - \alpha_{i}^T K_{i+1} =0 ;$$
if $s '  = \tilde t =0$,
$$K_{i}  \alpha_{{n-i-1}}'  =  \gamma_{i}^Tf_{i};$$
 if $\tilde s  = t' =0$,
 $$-(\alpha_{i}')^TK_{i+1 }=  \gamma_{n-i-1}f_{i}.$$

 \end{proof}
 
 \begin{lemma} 
 The component  $ f_{i} ((\eta-\beta_{i})^{-1})^T \gamma_{i}^T\gamma_{n-i-1} (\eta-\beta_{n-i-1} )^{-1}$ of the pairing $<,>_{i}$ can be written as
\begin{align*}
& [((\eta-\beta_{i})^{-1})^T  K_{i} \alpha_{n-i-1} - K_{i} \alpha_{n-i-1}(\eta-\beta_{n-i-1} )^{-1}]\\
=& [((\eta-\beta_{i})^{-1})^T  K_{i} \alpha_{n-i-1} -  \alpha_{i}^T K_{i+1}(\eta-\beta_{n-i-1} )^{-1}].
\end{align*}
 \end{lemma}
 \begin{proof}  Lemma \ref{adjoint} gives 
 $$\beta^T_{i}K_{i} \alpha_{n-i-1}= K_{i}\beta_{n-i}\alpha_{n-i-1}.$$
 From the commutativity of the diagrams giving the resolution of the sheaves, one has 
 $\beta_{n-i}\alpha_{n-i-1} = \alpha_{n-i-1} \beta_{n-i-1} + \alpha'_{n-i-1} \gamma_{n-i-1}$, and so
 $$\beta^T_{i}K_{i} \alpha_{n-i-1} = K_{i} (\alpha_{n-i-1} \beta_{n-i-1} + \alpha'_{n-i-1} \gamma_{n-i-1}),$$
 giving by the relations in the previous lemma
 $$\beta^T_{i}K_{i} \alpha_{n-i-1}  -  K_{i} \alpha_{n-i-1} \beta_{n-i-1} =   f_{i} \gamma_{i}^T \gamma_{n-i-1},$$
 substituting this into the   component:
  $$ [((\eta-\beta_{i})^{-1})^T [(-\eta + \beta^T_{i})K_{i} \alpha_{n-i-1} - K_{i}  \alpha_{n-i-1} (-\eta+ \beta_{n-i-1})] (\eta-\beta_{n-i-1} )^{-1}],$$
 and so 
 $$[(-K_{i} \alpha_{n-i-1} ( \eta- \beta_{n-i-1})^{-1} + ((\eta-\beta_{i})^{-1})^TK_{i}  \alpha_{n-i-1}]. $$
\end{proof}
 
Once one has the formula expressing the scalar product of the $P$ in terms of those on the $Q$, one can use it in the opposite direction. Suppose that we have pairings
$$H^0(\bP^1, Q_{n-i}^{i}) \otimes H^0(\bP^1, Q^{n-i}_{i})\rightarrow \bC,$$
that are such that the adjointness properties of Lemma~\ref{adjoint} are satisfied. One can then use the formulae of Lemma~\ref{PtoQ2} to define pairings on sections of the $P_i$. One can then show that it descends  from sections to the actual sheaves, and then gives a holomorphic pairing on sections of $E$.

There remains the issue of understanding how the pairing 
$H^0(\bP^1, Q_{0}^{n}) \otimes H^0(\bP^1, Q^{0}_{n})\rightarrow \bC$ interacts with the chain of correspondences 
$$ Q_{0}^{n}=Q^{0}_{n, k}\leftrightarrow Q^{0}_{n, k-1}\leftrightarrow Q^{0}_{n, k-2}\leftrightarrow...\leftrightarrow Q^{0}_{n, 1}\leftrightarrow Q^{0}_{n,0} = Q^{0}_{n}$$

Recall that $Q^{0}_{n, j}  = R^1(\pi_2)_*(\tilde E_n^0(\sum_{i\leq j}\tilde D_{i,\xi})) $.
The map $M_\psi: H^0(\bP^1, Q_{0}^{n}) \rightarrow H^0(\bP^1, Q^{0}_{n})$ is the map induced on cohomology by $m_\psi: \tilde E_0^n \rightarrow \tilde E^0_n$, that is multiplication by $\psi$. We have an isomorphism away from $\eta = \infty$, given by multiplication by $\xi$,   $m_\xi:\tilde E^0_n (\sum_i\tilde D_{i,\xi})\rightarrow \tilde E_0^n$;  composing, our map  $M_\psi$ becomes  induced by the multiplication map 
$$m_ {\prod_j(\eta-z_j)} : \tilde E^0_n (\sum_i\tilde D_{i,\xi})\rightarrow \tilde E^0_n;$$
 this factors into a composition of terms 
 $$m_ { (\eta-z_j) }: \tilde E_n^0(\sum_{i\leq j+1}\tilde D_{i,\xi})\rightarrow \tilde E_n^0(\sum_{i\leq j}\tilde D_{i,\xi})$$
(the actual function depends on the trivialisations chosen; these are more convenient than the functions $\xi_j$) and on the level of cohomologies:
 $$M_{\psi, j}: H^0(\bP^1, Q^{0}_{n, j+1})\rightarrow H^0(\bP^1, Q^{0}_{n,j}),$$
with the composition  $M_{\psi, 1}\circ M_{\psi, 2} \circ ... \circ M_{\psi, k} $ being $M_{\psi}$. The sheaves  $Q_0^n, Q^0_n, Q^{0}_{n,j}$ are supported at $\eta = z_i, i=1,..k$ and at other points $y_\mu$. Now split these terms according to their support:
$$Q^{0}_{n,j} = \oplus_i Q^{0}_{n,j, z_i} \oplus  Q^{0}_{n,j, rest}$$
where the last term is the piece supported at the other points $y_\mu$. As noted above the pairing localizes, so that only sections with overlapping support pair non-trivially: the pairing between $ Q_{0}^{n}$ and  $Q^{0}_{n}$ can be written as a sum of pairings
\begin{align}
H^0(\bP^1,Q_{0,z_i}^{n})\otimes H^0(\bP^1,Q^{0}_{n, z_i})&\rightarrow \bC,\ i= 1,..,k,\\
H^0(\bP^1,Q_{0,rest}^{n})\otimes H^0(\bP^1,Q^{0}_{n, rest})&\rightarrow \bC.
\end{align}
The map $M_{\psi, j}$ respects these decompositions and is an isomorphism on the subspaces $H^0(\bP^1,Q^{0}_{n,j, rest})$ and   $H^0(\bP^1,Q^{0}_{n,j+1, z_i}), i\neq j$; by taking a composition of the maps $M_{\psi, j+1}\circ M_{\psi, 2} \circ ... \circ M_{\psi, k} $ on $H^0(\bP^1, Q^{0}_{n,z_{j}})$ and $M_{\psi,1}\circ M_{\psi, 2} \circ ... \circ M_{\psi, j} $ on $H^0(\bP^1, Q^{0}_{n,j, z_j})$, one can then transfer the pairings $H^0(\bP^1, Q_{0, z_j}^{n}) \otimes H^0(\bP^1, Q^{0}_{n,z_j})\rightarrow \bC$ to 
$$H^0(\bP^1, Q_{0,j+1, z_{j }}^{n}) \otimes H^0(\bP^1, Q^{0}_{n,j,z_j})\rightarrow \bC.$$ As the map $M_{\psi, j}$ is induced by multiplication by a function, it is then self adjoint, by the same argument given  in Lemma (\ref{adjoint}) for $M_{\psi}$. Similar considerations also apply to the maps $M_{\xi, j}$.

\newcommand{\etalchar}[1]{$^{#1}$}
\providecommand{\bysame}{\leavevmode\hbox to3em{\hrulefill}\thinspace}
\providecommand{\MR}{\relax\ifhmode\unskip\space\fi MR }
\providecommand{\MRhref}[2]{%
  \href{http://www.ams.org/mathscinet-getitem?mr=#1}{#2}
}
\providecommand{\href}[2]{#2}

 \end{document}